\documentclass[reqno,11pt]{amsart}
\usepackage{amsmath,amssymb,latexsym,esint,cite,mathrsfs, tikz}
\usepackage{verbatim,wasysym}
\usepackage[left=3cm,right=3cm,top=2.7cm,bottom=2.7cm]{geometry}
\usetikzlibrary{decorations.pathreplacing, calligraphy}

\usepackage{microtype}
\usepackage{color,enumitem,graphicx}
\usepackage[colorlinks=true,urlcolor=blue, citecolor=red,linkcolor=blue,
linktocpage,pdfpagelabels, bookmarksnumbered,bookmarksopen]{hyperref}
\usepackage[hyperpageref]{backref}
\usepackage[english]{babel}
\usepackage[colorinlistoftodos]{todonotes}

\newtheorem{theorem}{Theorem}[section]
\newtheorem{lemma}[theorem]{Lemma}
\newtheorem{e-proposition}[theorem]{Proposition}
\newtheorem{corollary}[theorem]{Corollary}

\newtheorem{proposition}[theorem]{Proposition}
\newtheorem{definition}[theorem]{Definition}
\newtheorem{remark}[theorem]{Remark}
\newtheorem{ex}[theorem]{Example}

\newcommand{\R}{\mathbb{R}}

\renewcommand{\S}{\mathbb{S}}

\newcommand{\eps}{\varepsilon}
\newcommand{\beq}{\begin{equation}}
\newcommand{\eeq}{\end{equation}}

\DeclareMathOperator{\Div}{div}

\DeclareMathOperator*{\esssup}{ess ~sup}         %

\newcommand{\RN}{\mathbb{R}^N}

\newcommand{\Om}{\Omega}

\renewcommand{\l}{\left}
\renewcommand{\r}{\right}
\def\Xint#1{\mathchoice
{\XXint\displaystyle\textstyle{#1}}%
{\XXint\textstyle\scriptstyle{#1}}%
{\XXint\scriptstyle\scriptscriptstyle{#1}}%
{\XXint\scriptscriptstyle\scriptscriptstyle{#1}}%
\!\int}
\def\XXint#1#2#3{{\setbox0=\hbox{$#1{#2#3}{\int}$ }
\vcenter{\hbox{$#2#3$ }}\kern-.6\wd0}}

\def\intmed{\Xint-}

 \DeclareMathOperator*{\essliminf}{ess\,lim\,inf}
\DeclareMathOperator*{\esslimsup}{ess\,lim\,sup}

\numberwithin{equation}{section}

\title[Lipschitz regularity for solutions of a general class of elliptic equations]{Lipschitz regularity for solutions of a general class of elliptic equations}

\author[G.\ Marino]{Greta Marino}
\author[S. \ Mosconi]{Sunra Mosconi}

\address[G.\ Marino]{Institut f\"ur Mathematik
\newline\indent
Universit\"{a}t Augsburg 
\newline\indent 
Universit\"{a}tsstra\ss e 12a, 86159 Augsburg, Germany}
\email{greta.marino@uni-a.de}

\address[S.\ Mosconi]{Department of Mathematics and Computer Science
	\newline\indent
	University of Catania
	\newline\indent
	Viale A. Doria 6, I-95125 Catania, Italy}
\email{sunra.mosconi@unict.it}


\begin{document}

\begin{abstract}
We prove local Lipschitz regularity for local minimisers of 
\[
W^{1,1}(\Omega)\ni v\mapsto \int_\Omega F(Dv)\, dx
\]
where $\Omega\subseteq \R^N$, $N\ge 2$ and $F:\R^N\to \R$ is a quasiuniformly convex integrand in the sense of Kovalev and Maldonado \cite{KoMa}, i.\,e.\,a convex $C^1$-function such that the ratio between the maximum and minimum eigenvalues of $D^2F$ is essentially bounded. This class of integrands includes the standard singular/degenerate functions $F(z)=|z|^p$ for any $p>1$ and arises naturally as the closure, with respect to a natural convergence, of the  strongly elliptic  integrands of the Calculus of Variations.
  \end{abstract}

\maketitle

	\begin{center}
		\begin{minipage}{9cm}
			\small
			\tableofcontents
		\end{minipage}
	\end{center}
 
 \section{Introduction}
 \subsection{Overview and main result}
Consider a convex and coercive  $F:\R^N\to \R$. The aim of this paper is to prove Lipschitz regularity of the local minimisers of the standard integral of Calculus of Variations
\beq
\label{i0j}
J(u, \Omega)=\int_\Omega F(Du)\, dx
\eeq
i.\,e.\,of those $u\in W^{1,1}_{\rm loc}(\Omega)$ such that $F(Du)\in L^1_{\rm loc}(\Omega)$ and for any open ${\mathcal O}\Subset\Omega$ it holds
\[
J(u, {\mathcal O})\le J(u+w, {\mathcal O})\qquad \forall \, w\in W^{1,1}_0({\mathcal O}).
\]
The integrands $F$ we are interested in satisfy for some $H<\infty$
\beq
\label{i01}
 \lambda_{\rm max} (D^2F(z))\le H\, \lambda_{\rm min}(D^2F(z))\qquad \text{for a.\,e.\,$z\in \R^N$},
\eeq
where here and in the following $\lambda_{\rm max}(M)$ and $\lambda_{\rm min}(M)$ denote respectively the minimum and maximum eigenvalues of the symmetric matrix $M$. In this introductory paragraph, the integrands obeying \eqref{i01} will be called {\em uniformly elliptic} (even if the term is ubiquitous and thus may cause some confusion). As will be clarified later, most of  the available Lipschitz regularity results for  local minimisers prescribe that the singular set, where $\lambda_{\rm max}(D^2F(z))$ blows up, and the degeneracy set, where  $\lambda_{\rm min}(D^2F(z))$ vanishes, are both bounded. As will be clarified in the following discussion, condition \eqref{i01} alone instead allows for both sets to  be  simultaneously  unbounded (even dense) and sizeable  (in the sense of Hausdorff dimension).

In order to motivate the uniform ellipticity condition \eqref{i01}, let us review some well known facts regarding local minimisers of \eqref{i0j}.
If $F$ fulfils a polynomial upper bound, these are finite energy solutions of the corresponding Euler-Lagrange equation
\beq
\label{i00}
{\rm div}\, (DF(Du))=0
\eeq
in $\Omega$.
Existence of minimisers is ensured by standard methods once a superlinearity condition on $F$ at $\infty$
is imposed, however well known examples (see \cite{Gi, Ma0}) show that under these s\^ole conditions (sub-polynomial growth and superlinearity) they may fail to be locally bounded. Regularity of local minimisers is a classical topic dating back to Hilbert's XIX problem, solved through Schauder and DeGiorgi-Nash-Moser theories under the {\em strong ellipticity} assumption
\[
0<\inf_{z\in \R^N} \lambda_{\rm min}(D^2F(z))\le \sup_{z\in \R^N} \lambda_{\rm max}(D^2F(z))<\infty.
\]

Since the solution of Hilbert's problem, much effort has been dedicated to weaken this assumption. A natural quantity arising in the aforementioned regularity theories is the {\em ellipticity ratio} (also called  the {\em linear dilatation} of $D^2F(z)$), namely
\[
{\rm e}(z)=\frac{\lambda_{\rm max}(D^2F(z))}{\lambda_{\rm min}(D^2F(z))}.
\]

The closure of smooth, strongly elliptic integrands  with respect to pointwise (or, equivalently,  $C^1_{\rm loc}(\R^N)$) convergence, when coupled with a uniform bound on the ellipticity ratio, turns out to be a cone consisting of
\begin{itemize}
\item
Affine functions, which are the extremals of the cone
\item
Convex superlinear $F\in C^1_{\rm loc}(\R^N)\cap W^{2, N}_{\rm loc}(\R^N)$ obeying \eqref{i01} a.\,e.\, for some $H<\infty$
\end{itemize}
(see Corollary \ref{cclosure}).
The second type of integrands have been studied in the seminal papers \cite{K, KoMa} (to which we refer for further details) and are called $H$-{\em quasiuniformly convex},\footnote{Actually, in \cite{K, KoMa} emphasis is given to the geometric quasiconformality constant of the gradient map, rather than on its linear dilatation $H$ defined in \eqref{i01}}  henceforth abbreviated by $H$-q.\,u.\,c.. They are related to the well developed theory of quasiconformal maps, as their gradient mapping is indeed quasiconformal. 
It turns out, quite conveniently we may say, that in dimension $N\ge 2$ condition \eqref{i01} implies not only superlinearity, but also  sub-polynomial growth and strict convexity of the integrand $F$, granting existence of local minimisers (or actually, minimisers under various boundary conditions) in $W^{1, 1+1/H}_{\rm loc}(\Omega)$. Moreover, local minimisers and finite energy weak solutions of \eqref{i00} coincide (see Proposition \ref{minimi}). Equation \eqref{i00} can be formally derived with respect to each variable $x_\alpha$, giving for each partial derivative $\partial_\alpha u$
\beq
\label{i06}
{\rm div}\, (D^2F(Du)\, D\partial_\alpha u)=0,\qquad \alpha=1,\dots, N,
\eeq
which fails to be strongly elliptic for $Du$ belonging to ${\rm Sing}_F\cup {\rm Deg}_F$, where
 \beq
\label{sing-deg}
 \begin{split}
{\rm Sing}_F&=\left\{\bar z\in \R^N: \esslimsup_{z\to \bar z}\lambda_{\rm max}(D^2F(z))=\infty\right\},\\
{\rm Deg}_F&=\left\{\bar z\in \R^N: \essliminf_{z\to \bar z}\lambda_{\rm min}(D^2F(z))=0\right\}.
\end{split}
\eeq
If $F$ is q.\,u.\,c., the previous sets are related to the so-called quasiconformal $\infty$- and $0$- sets respectively, appearing in the quasiconformal Jacobian problem \cite{BHS}. These can be quite large, as shown in \cite[Section 4 and 5]{KoMa}: given an arbitrary  $E\subseteq \R^N$ of Hausdorff dimension less than $1$, there exist a  q.\,u.\,c.\,$F_1$ such that $E\subseteq {\rm Sing}_{F_1}$ and a q.\,u.\,c.\,$F_2$ such that $E\subseteq {\rm Deg}_{F_2}$. In particular, both ${\rm Sing}_F$ and ${\rm Deg}_F$ can be dense in $\R^N$ for $N\ge 2$. On the other hand, for $F$ q.\,u.\,c., neither ${\rm Sing}_F$ nor ${\rm Deg}_F$ can contain rectifiable curves (see \cite{BHS, KoMa}).

Despite the natural appearance of q.\,u.\,c\,integrands as limits (in the sense described above) of strongly elliptic ones, the corresponding regularity theory for \eqref{i06} seems more delicate. Indeed, the $L^\infty$ and $C^\alpha$ bounds for $Du$ when $F$ is strongly elliptic essentially depend on the quantity
\[
\frac{\sup_\Omega \lambda_{\rm max} (D^2F(Du))}{\inf_{\Omega} \lambda_{\rm min}(D^2F(Du))}
\]
rather than on the actually controlled quantity
\[
\sup_\Omega\frac{ \lambda_{\rm max} (D^2F(Du))}{\lambda_{\rm min}(D^2F(Du))}
\]
and a naive limiting argument is bound to fail.

Nevertheless, some classes of q.\,u.\,c.\,where the standard regularity theory can still produce results are already well known. An important type of q.\,u.\,c.\,integrands are for instance  those with {\em Uhlenbeck structure}, i.\,e.\,depending only on the modulus of the gradient. Given a non-decreasing $G:[0, \infty[\to [0, \infty[$, the function $F(z)=G(|z|)$ is q.\,u.\,c.\,if and only if $G\in C^1([0, \infty))\cap W^{2, \infty}_{\rm loc}(\R_+)$ and there exists a constant $C>0$ such that 
\beq
\label{i09}
\frac{1}{C}\le \frac{t\, G''(t)}{G'(t)} \le C \qquad \text{for a.\,e.\,$t>0$}.
\eeq
In particular, for a q.\,u.\,c.\,function with Uhlenbeck structure, it always holds
\beq
\label{isd}
{\rm Sing}_F\subseteq \{0\}\quad \text{as well as} \quad  {\rm Deg}_F\subseteq \{0\}
\eeq
which, since obviously ${\rm Sing}_F\cap {\rm Deg}_F=\emptyset$ for a q.\,u.\,c.\,$F$, justifies the traditional dichotomy between the singular or degenerate case of \eqref{i00}. 

The regularity theory for solutions of \eqref{i00} when $F$ is a  Uhlenbeck q.\,u.\,c.\,integrand is well developed, even when non-homogenous terms with low summability are included and no variational setting it available. See e.\,g.\,\cite{CM} and the literature therein for a survey on  the available results for solutions of
\beq
\label{icm}
{\rm div}\, \left(\frac{G'(|Du|)}{|Du|}\, Du\right)=f
\eeq 
up to the Lipschitz scale, depending on the summability properties of $f$.  Regarding higher regularity, in \cite{L} minimisers have been proved to be $C^{1,\alpha}$  for integrands which are even more general than the Uhlenbeck ones, but still fulfil a uniform {\em radial bound} on the ellipticity ratio. By this we mean that the main assumption for the $C^{1,\alpha}$ regularity is the existence of two radial functions  $\lambda_{\min}, \lambda_{\rm max}:\R_+\to \R_+$ obeying
\[
\lambda_{\rm min}(|z|)\le \lambda_{\rm min}(D^2F(z)),\qquad \lambda_{\rm max}(|z|)\ge \lambda_{\rm max} (D^2F(z))
\]
and 
\[
\sup_{t\in \R_+} \frac{\lambda_{\rm max}(t)}{\lambda_{\rm min}(t)}<\infty
\]
so that these integrands are q.\,u.\,c.. In this framework both ${\rm Sing}_F$ and ${\rm Deg}_F$ are anyway restricted to have radial symmetry, but since  neither of those  can contain circles, \eqref{isd} holds again. 

Notice that when $F$ is a q.\,u.\,c., normalised in such a way that $\min_{\R^N} F= F(0)=0$ (which can always be safely assumed), both $F$ and $DF$ actually fulfil a two-sided isotropic control of the form
\[
\frac{1}{C}\, A(|z|)\le F(z)\le C\, A(|z|), \qquad \frac{1}{C}\, A'(|z|)\le |DF(z)|\le C\, A'(|z|)
\]
for a $C^1$ Young function $A:[0, \infty[\to [0, \infty[$ (see Section \ref{3}) but, as mentioned before, no such isotropic control is available at the level of the ellipticity ratio for a general q.\,u.\,c.\,integrand.  

Our interest in this framework is motivated by the results in \cite[Chapter 16]{AIMbook} where regularity theory for the corresponding minimisers has been successfully developed in two space dimensions. More precisely, in \cite[Theorem 16.4.5]{AIMbook}, the authors prove $C^{1,\alpha}$ regularity of finite energy\footnote{In this possibly non-variational setting, the energy considered is given by the integral of $({\mathcal A}(Du), Du)$.}
solutions of 
\beq
\label{i03}
{\rm div}\, {\mathcal A}(Du)=0
\eeq
in the plane, under the even more general assumption that the mapping ${\mathcal A}:\R^2\to \R^2$ fulfils
\[
\left({\mathcal A}(z)-{\mathcal A}(w), z-w\right)\ge \delta\, |{\mathcal A}(z)-{\mathcal A}(w)|\, |z-w|, \qquad \forall \, z, w\in \R^2
\]
for some $\delta>0$. Such mappings are called {\em $\delta$-monotone} and have been introduced in \cite{K}, where it is shown that any gradient map of a quasiuniformly convex integrand is $\delta$-monotone, but notice that in general equation \eqref{i03} may be non-variational. Hence the results of \cite{AIMbook} in the plane are more general and stronger than ours, but the method developed therein is based on the reduction of \eqref{i03} to a first-order system  of Beltrami type for the complex gradient of $u$, and thus are constrained to the two-dimensional setting.

We can now state our main result.

\begin{theorem}\label{imt}
For $N\ge 2$, let  $F\in C^1_{\rm loc}(\R^N)\cap W^{2,1}_{\rm loc}(\R^N)$ be convex, obey \eqref{i01} a.\,e.\,and fulfil the normalisation
\beq
\label{inorm}
F(z)\ge F(0)=0\qquad \forall \, z\in \R^N.
\eeq
Then any local minimiser $u$  for $J$ in $\Omega$ is locally Lipschitz. In particular, there exists a constant $C=C(H, N)>0$ such that if $B_{2R}\subseteq \Omega$, then
\beq
\label{ilest}
	\sup_{B_{R}}|F(Du)|\le C\,  \intmed_{B_{2R}}F(Du)\, dx.
\eeq
\end{theorem}

\begin{remark}[Comments on the statement]
\ 
\begin{itemize}
\item
Except for the normalisation condition \eqref{inorm}, the assumptions on $F$ can be equivalently stated as $F$ being $H$-quasiuniformly convex, or $DF$ being a quasiconformal map with maximal linear dilatation bounded by $H$. Notice that, thanks to the quasiconformality of $DF$, the $W^{2,2}_{\rm loc}(\R^N)$ regularity of $F$ together with condition \eqref{i01} automatically improves to $F\in W^{2,N+\eps}_{\rm loc}(\R^N)$ for some $\eps>0$.
\item
The normalisation condition \eqref{inorm} is made only to have the cleaner estimate \eqref{ilest}, as any $H$-q.\,u.\,c.\,integrand $F$ has a unique minimum point $\bar{z}$. Then, the integrand
\[
\tilde{F}(z)=F(z+\bar z)-F(\bar z)
\]
is still $H$-q.\,u.\,c.\,and the function $\tilde{u}(x)=u(x)-(x, \bar z)$ is a local minimiser for the corresponding integral functional. Notice that \eqref{ilest}, having no additional term on the right, prescribes point-wise smallness of $Du$ for small values of its energy.
\item
Condition \eqref{i01} giving $H$-q.\,u.\,convexity does not distinguish between singular or degenerate equations, since both $z\mapsto |z|^{1+1/H}$ and $z\mapsto |z|^{1+H}$ are $H$-q.\,u.\,c.. More substantially, 
a simple construction  (see Example \ref{exsd}) shows that for a q.\,u.\,c.\,integrand $F$, both ${\rm Sing}_F$ and ${\rm Deg}_F$ can be simultaneously non-empty, so that the traditional dichotomy between singular/degenerate equation cannot possibly hold in this setting. 
\item
We could also have treated local minimisers of
\[
J(u, {\mathcal O})=\int_{\mathcal O} F(Du)+f\, u\, dx, \qquad {\mathcal O}\Subset\Omega,
\]
for sufficiently smooth $f$. However, our method seems to provide Lipschitz regularity under non-optimal regularity conditions on $f$. The natural condition ensuring boundedness of $Du$ should be $f\in L^q_{\rm loc}(\Omega)$ for some $q>N$, but it appears that substantial modifications to the proof are needed to obtain such a result.
\end{itemize}
\end{remark}

\subsection{The r\^ole of the Uhlenbeck structure}

Let us discuss briefly the r\^ole of a radial control on the ellipticity ratio in proving a-priori gradient bounds for finite energy solutions of \eqref{i00} with a convex $F$. The standard strategy to this end, which goes back to Bernstein,  
is to find a suitable coercive function $G:\R^N\to \R$ and  coefficients $a_{i j}:\R^N\to \R$ such that 
\begin{enumerate}
\item
The matrix $A(z)=(a_{ij}(z))$ has controlled ellipticity for every $z\in \R^N$;
\item
The function $G(Du)$  solves
\beq
\label{i0t}
{\rm div}\, (A(Du)\, D(G(Du)))\ge 0.
\eeq
\end{enumerate}
Once these tasks are achieved, {\em linear} regularity theory can be applied to $G(Du)$. If, for instance, $A$ turns out to be strongly elliptic, then solutions of \eqref{i0t} fulfil
\[
\|G(Du)\|_{L^\infty(B_R)}\le C\, \intmed_{B_{2R}} G(Du)\, dx
\]
 and we are reduced to control the integrand on the right  (or variants of it) by $F(Du)$. 
More generally, once suitable refinements of the linear elliptic regularity step are developed, this scheme is flexible enough to deal with choices of  $A$ which may not be strongly elliptic, as long as their ellipticity ratio does not blow up too fast for $z\to \infty$ (see e.\,g.\,\cite{BM} and the literature therein for the non-uniformly elliptic setting). 

For the sake of this discussion, we will suppose that $F$ and $u$ are smooth and that we have chosen $\lambda_{\rm min}, \lambda_{\rm max}:\R^N\to \R_+$ in such a way that for all $z\in \R^N$ it holds 
\beq
\label{irad}
0\le \lambda_{\rm min}(z)\le \lambda_{\rm min}(D^2F(z)), \qquad \lambda_{\rm max}(z)\ge \lambda_{\rm max}(D^2F(z)), \qquad \frac{\lambda_{\rm max}(z)}{\lambda_{\rm min}(z)}\le H.
\eeq
The standard approach to construct  the couple $(A, G)$ is to multiply each equation in \eqref{i06} by $\partial_{\alpha}u$ and sum, to obtain
\beq
\label{i2}
{\rm div}\, \Big(D^2F(Du)\, \sum_{\alpha=1}^N (D\partial_\alpha u) \, \partial_\alpha u\Big)=\sum_{\alpha=1}^ND^2F(Du)\, D\partial_\alpha u\, D\partial_\alpha u\ge 0
\eeq
where the last inequality follows from the convexity of $F$. With this choice, however, we are committing to a radial control on the ellipticity ratio, since
\[
\sum_{\alpha=1}^N (D\partial_\alpha u)\, \partial_\alpha u=D \frac{|Du|^2}{2}
\]
and a natural choice for $A$ is
\[
A(z)=\frac{D^2F(z)}{\lambda_{\rm min}(z)},
\]
which is strongly elliptic if the ellipticity ratio is uniformly bounded, while in the meantime \eqref{i2} reads
\[
{\rm div}\, \left(A(Du)\, \lambda_{\rm min}(Du)\, D \frac{|Du|^2}{2}\right)\ge 0.
\]
In order to construct $G$ we must impose
\beq
\label{i3}
DG(z)=\lambda_{\rm min}(z)\, D \frac{|z|^2}{2}=\lambda_{\rm min}(z) \, z
\eeq
but this relation forces $G$, and thus $\lambda_{\rm min}$, to be radial functions.
If this is so, i.\,e.\,$\lambda_{\rm min}(z)=\lambda_{\rm min}(|z|)$, we can indeed set
\[
G(z)=\int_0^{|z|}\lambda_{\rm min}(s)\, s\, ds
\]
which fulfils \eqref{i3}.
Notice finally that, since $F(0)=0$ and $DF(0)=0$, for $|\omega|=1$ and $t>0$ it holds
\[
F(t\, \omega)=\int_0^{t}\left(D^2F(s\, \omega) \, \omega, \omega\right)(t-s)\, ds\ge \int_0^{t}\lambda_{\rm min}(s)\, (t-s)\, ds\ge \int_0^{t/2}\lambda_{\rm min}(s)\, s\, ds=G(t/2)
\]
which can be used to bound $G(Du)$ in terms of the natural integrand $F(Du)$.\footnote{As we will see, if ${\rm e}$ is uniformly bounded it always holds $F(2\, z)\le C\, F(z)$.}
Without a radial control on $\lambda_{\rm min}$ or $\lambda_{\rm max}$, however, the previous approach fails. 

\begin{remark}
It may be worth noting that multiplying \eqref{i06} by $\partial_\alpha F(Du)$ and summing up, one finds
\[
{\rm div}\, \big(D^2F(Du) D (F(Du))\big)={\rm Tr}\, \big(D^2F(Du)\, D^2u\, D^2F(Du)\, D^2u).
\]
As will be quantified in the following paragraph, the convexity of $F$ ensures that the right-hand side above is still nonnegative, so that we have found another couple $(A, G)$ fitting the previous scheme, namely,
\[
A= D^2F(Du), \qquad G=F.
\]
The issue, however, is that $A$ fails to be  strongly elliptic and is only uniformly elliptic. Needless to say, the linear theory for solutions of 
\[
{\rm div}\, (A(x)\, Dv)=0
\]
under the uniform ellipticity condition 
\[
0\le \lambda_{\rm max}(A(x))\le H\, \lambda_{\rm min}(A(x))\qquad \text{a.\,e.}
\]
is very poor (as can be seen by simple one-dimensional examples) and does not provide boundedness of solutions. Nevertheless, estimate \eqref{ilest} says that the function $v=F(Du)$ behaves ``as if" the operator ${\rm div}\, \big(D^2F(Du)\, Dv)$ is strongly elliptic on $v$.
\end{remark}

\subsection{Outline of the proof}
The possibility of tackling the regularity problem for finite energy solutions of \eqref{i00}
 in dimension $N\ge 3$  under the general uniform ellipticity condition \eqref{i01} has been considered for the first time in \cite{GM}. The approach adopted therein can be seen in  the framework of {\em nonlinear differential inclusions}, originally rooted in the works   \cite{B, T, DP}   and well developed nowadays, thanks to the significant applications investigated in \cite{MS, KMS, VS, DS}, to name a few. Specifically, one can look at \eqref{i00} focusing only on the {\em stress field} $DF(Du)=V$. Formally deriving the equation, we find the system
 \beq
 \label{i04}
 {\rm Div}\, (DV^t)=0
 \eeq
 where ${\rm Div}$ is the row-wise divergence operator acting on matrix-valued functions and $M^t$ denotes the transpose of the matrix $M$. Notice that, contrary to the Laplacian operator
 \[
 \Delta V={\rm Div}\, (DV),
 \]
 the operator in \eqref{i04} is far from being elliptic: its kernel contains all compactly supported solenoidal vector fields and there is no hope to prove regularity of solutions to system \eqref{i04}. Ellipticity, however, can be restored by constraining $DV$ to pointwise belong to a suitable cone. The natural and fruitful one in our framework, as found in \cite{GM}, is given by
\[
 {\mathcal K}_H=\left\{M\in \R^{N\times N}: {\rm Tr}\, (M\, M)\ge \frac{1}{H}\, {\rm Tr}\, (M\, M^t)\right\},
\]
 for   $H$ given in \eqref{i01}. The fact that $DV$ must belong to such a cone depends only on the structural condition 
 \[
 DV= D^2F(Du)\, D^2u,
 \]
i.\,e.\,$DV$ must be the product of a positive definite symmetric matrix  with linear dilatation bounded by $H$ with a symmetric matrix (see Lemma \ref{lemma51} below). Note that, as $H\searrow 1$, we have  ${\mathcal K}_H\searrow {\mathcal K}_1={\rm Sym}_N$, that is the linear space of symmetric $N\times N$ matrices, hence the solutions of the differential inclusion
\beq
\label{i05}
\begin{cases}
 {\rm Div}\, (DV^t)=0\\
 DV\in {\mathcal K}_1
 \end{cases}
 \eeq	
are just the gradients of harmonic functions.

 For $H>1$,  system \eqref{i05} is elliptic at the $L^2$ level, but this kind of restored ellipticity  has a limited effect at finer regularity scales.
Indeed, stress fields of $p$-harmonic equation certainly solve the previous nonlinear differential inclusion for $p=1+1/H$, and for $p\ne 2$ there exists (see \cite{IM}) in the plane a $p$-harmonic function whose stress field vanishes only at the origin and is homogenous of degree
\[
d=\frac{1}{6}\left(p+\sqrt{p^2+12\, p-12}\right)=\frac{1}{6\, H}\left(H+1+\sqrt{H^2+14\, H+1}\right). 
\]
Since the last expression is always less than $1$ for $H>1$, the best regularity one can expect from solutions of \eqref{i05} is at most $C^{\alpha}$, for $\alpha=\alpha(H)\in \ [1/3, 1[$ and for any $H>1$.  We stress here that we were {\em not} able to prove that solutions of \eqref{i05} are bounded, which would imply the qualitative part of Theorem \ref{imt} and would be meaningful in light of current lines of research on differential inclusions (see e.\,g.\,\cite{ADHRS, GRS}). In fact, we had to resort to additional structure possessed by the original equation \eqref{i00}. More precisely, by still denoting $V=DF(Du)$ for a solution of \eqref{i00} and by using the differential inclusion $DV\in {\mathcal K}_H$, we prove a family of {\em Caccioppoli inequalities} of the form
\beq
\label{icac}
\int_{B_r\cap \{F(Du)>k\}} |DV|^2\, dx\le \frac{C_H}{(R-r)^2}\int_{B_R\cap \{F(Du)>k\}}|V|^2\, dx
\eeq
for arbitrary $k\in \R$ and for $R> r>0$.
Then we translate this vectorial Caccioppoli inequality into a family of Caccioppoli inequalities on suitable scalar functions intrinsically defined as suitable Minkowski functionals associated to 
\[
G(z)=F(DF^{-1}(z))
\]
(which is well defined since $DF$ is invertible), namely,
\beq
\label{idgk}
g_k(z)=\inf \left\{t>0: G(z/t)>k\right\}
\eeq
for any $k\in \R$. Notice that, since $V=DF(Du)$, then  
\[
\{F(Du)>k\}=\{G(V)>k\}=\{g_k(V)>1\}
\]
and the quasiuniform convexity of $F$ is pivotal to prove that \eqref{icac} translates to
\beq
\label{icac2}
\int_{B_r} |D(g_k(V)-1)_+|^2\, dx\le \frac{C_H}{(R-r)^2}\, \int_{B_R} |g_k(V)|^2\, dx.
\eeq
Additional fine properties of the family of $1$-homogeneous functions $\{g_k\}$ allow to adapt the classical De Giorgi method for proving boundedness of subsolutions.  Notice that \eqref{icac2} exhibits two main differences with respect to a standard Caccioppoli inequality. On the one hand, due to the vectorial nature of \eqref{icac}, its dependence on the level $k$ is encapsulated rather implicitly in the family of $1$-homogeneous functions $\{g_k\}$ rather than directly on $V$. More substantially, its right-hand side is quite bigger than the one usually found for scalar problems, which would provide an integrand of the form $(g_k(V)-1)_+$ on the right instead of $g_k(V)$. Nevertheless, the De Giorgi method can still be adapted to this weaker setting, providing an $L^\infty$ bound on $F(Du)$ in terms of the $L^2$ norm of $V$. It is quite fortunate that this estimate implies, through a refinement of the results in \cite{GM},  the natural bound \eqref{ilest}.

\subsection{Related results}

In the uniformly elliptic setting the result in Theorem \ref{imt} has, as already remarked, been previously obtained in \cite{L} under an Uhlenbeck type control of the form \eqref{irad}, with $\lambda_{\rm min}(z)$ and $\lambda_{\rm max}(z)$ depending only on the modulus of $z$. In \cite{L}, actually, $C^{1, \alpha}$ regularity is proved under these assumptions. Lipschitz regularity has been proved in \cite{CM1} for solutions of \eqref{icm} coupled with Dirichlet or Neumann boundary conditions, with optimal  regularity assumptions both on $f$ and on the domain. The corresponding regularity for solutions of systems with Uhlenbeck structure is treated in \cite{CM2}.  

Notice that in \cite{L, CM, CM2} the authors assume that the function $G'$ appearing in \eqref{icm} (or in the lower/upper controls on $D^2F(z)$) belongs to $C^1(\R_+)$, while we admit $G'\in {\rm Lip}_{\rm loc}(\R_+)$ (this regularity seems optimal due to \cite[Example 3.5]{GM}). 
However, inspecting the proofs in \cite{CM, CM2} shows that the results therein hold under this more general assumption. 

Our result covers in particular the Finslerian anisotropic setting, which we will now briefly describe due to its relevance in recent research trends. The integrand considered in this framework are of the form 
\[
F(z)=G(h(z))
\]
where $G\in C^1(\R_+)\cap W^{2,\infty}_{\rm loc}(\R_+)$ is increasing, convex and fulfils \eqref{i09} (see \cite[Example 3.7]{GM}),
while $h$ is a $C^2(\R^N\setminus \{0\})$  convex,  positive, $1$-homogeneous function  (not necessarily symmetric) such that the principal curvatures of  $\partial \{h<1\}$ are bounded from below by a positive constant, see \cite[Example 3.7]{GM} for more details. 
The arguments  in \cite[Section 3]{CFV} show that $C^{1, \alpha}$ regularity holds true for the corresponding minimisers when $G$ has more stringent controls of $p$-growth type and $h\in C^{3, \alpha}(\R^N\setminus \{0\})$.

While the non-uniformly elliptic case is not treated in this research, it is worth mentioning that there are many instances where Lipschitz regularity can be obtained even if the ellipticity ratio is unbounded.  Generally speaking, in order to get Lipschitz continuity of local minimisers, one usually requires a growth control on the ellipticity ratio outside a bounded set, but in most cases the resulting Lipschitz bound critically depends on the diameter of the aforementioned set. The literature in this area is huge and, regarding local minimisers, we refer to \cite{BM} for some recent results and  for a rather comprehensive description of this research topic. Lipschitz regularity can be obtained for rather wild functionals by the so called Hilbert-Haar method, adapted to the Calculus of Variations by Stampacchia and Hartman \cite{S, HS}. This allows to obtain Lipschitz regularity of minimisers having a prescribed boundary value obeying the so-called bounded slope condition, under very loose conditions on the integrand. We refer to \cite{BB} and the literature therein for more details on this approach. Another class of non-uniformly elliptic integrands are the so-called {\em orthotropic} ones, where both ${\rm Sing}_F$ and ${\rm Deg}_F$ are unbounded,  being union of hyperplanes $\{z_i=0\}$. In this setting, assuming  a-priori boundedness  of the minimiser is the key to infer Lipschitz regularity.   We refer to \cite{BBL} and the literature therein  for further details on this class of integrands.

\subsection{Structure of the paper}
In Section \ref{2} we collect the properties of quasiuniformly convex functions which are relevant to the proof. Section \ref{3} is devoted to a refinement of the Sobolev regularity of the stress field $V=DF(Du)$, originally obtained in \cite{GM}. In Section \ref{4} we construct a sequence of approximating elliptic problems and corresponding solutions, allowing to reduce the proof of Theorem \ref{imt} to the smooth setting. In Section \ref{5} we prove the Caccioppoli inequality \eqref{icac}, while in Section \ref{6} we derive several properties of the family of  $1$-homogeneous functions defined in \eqref{idgk}. The final Section \ref{7} is devoted to the proof of Theorem \ref{imt}.

{\bf Acknowledgments.} \  We thank prof. L. Kovalev for improving a preliminary version of the paper. The authors are member of GNAMPA of INdAM.
G.\,M.\,acknowledges the support of DFG via grant GZ: MA 10100/1-1 project number 496629752.
S.\,M.\,is  partially supported by the  projects PIACERI linea 2  and linea 3 of the University of Catania and by the GNAMPA's projects {\em Equazioni alle derivate parziali di tipo ellittico o parabolico con termini singolari} and {\em Problemi ellittici e parabolici con termini di reazione singolari e convettivi}.

{\bf Notations.} \  In the whole paper we restrict to the case $N\ge 2$.  By $|v|$ we denote the Euclidean norm of a vector $v \in \RN$, while $(v, w)$ stands for the scalar product of any $v, w \in \RN$. Given a vector field, upper and lower indexes stand for its components and its derivatives, respectively.  We sum over repeated indexes.
 %, and ${\rm Id}$ is the identity function of $\RN$. 
 With the symbol $\Om$ we mean a bounded, open subset of $\RN$, while $B_r(x_0)$ denotes a  ball with  center $x_0 \in \RN$ and radius $r>0$, and by  $B_r$ we indicate a ball of radius $r$, not necessarily centred at the origin.
 %, and if $B= B_r(x_0)$ we set $\lambda B= B_{\lambda r}(x_0)$.  
For any measurable $E\subset \R^N$, $|E|$ denotes its $N$-dimensional Lebesgue measure. We will omit the domain of integration when it is the whole $\RN$, if this causes no confusion. 
Furthermore, for the sake of notational simplicity we set  $\|f\|_m:= \|f\|_{L^m(\RN)}$.

Let $M= (m_{ij})$ be an $N \times N$ matrix with real entries, and let $M^{t}$ denote its transpose. We consider the Frobenius norm
	\[
	|M|_2= \l(\sum_{ij=1}^N |m_{ij}|^2 \r)^{1/2}
	\]
arising from the scalar product $(M_1, M_2)_2= \operatorname{Tr}(M_1 \, M_2^{t})$. 
%For every $v, w \in \RN$ we define the matrices $v \otimes w= (v_i w_j)$ as well as $v \wedge w= v \otimes w- w \otimes v$. 
We further  denote by  ${\rm Id}$ the identity matrix.
% and by $O_N$ the orthogonal group of $N \times N$ matrices, i.\,e. those matrices such that $M M^{t}= M^{t} M= {\rm Id}$. 
Finally, for any  matrix $M$, $\sigma_{\rm max}(M)$ and $\sigma_{\rm min}(M)$ denote its maximum and minimum singular values (i.\,e.\,the square roots of the eigenvalues of $M\, M^t$), respectively. If $M$ is symmetric and non-negative definite, we will  use the notation $\lambda_{\rm max}(M)$ and $\lambda_{\rm min}(M)$ as in this case eigenvalues and singular values coincide.

\medskip

\section{Quasiuniformly convex integrands}
\label{2}

\begin{definition}

A map $\Phi:\R^N\to \R^N$ is {\em $K$-quasiconformal}  if it is a $ W^{1,1}_{\rm loc}(\R^N; \R^N)$ homeomorphism, it is a.\,e.\,differentiable and the inequality
\[
\frac{1}{K} \, \sigma_{\rm max}(D\Phi)^N\le |J\, \Phi|\le K\, \sigma_{\rm min}(D\Phi)^N
\]
holds a.\,e.\,in $\R^N$.
\end{definition}

Let $\overline{\R}^N$ denote the one-point compactification of $\R^N$. By \cite[Theorem 17.3]{V} any quasiconformal map  $\Phi:\R^N\to \R^N$ can be extended to a homeomorphism of $\overline{\R}^N$ by setting $\Phi(\infty)=\infty$, in the meantime keeping its (geometric) quasiconformality constant $K$ unaltered.
Moreover,  $\Phi\in W^{1,N}_{\rm loc}(\R^N)$ and is {\em quasisymmetric},  that is, there are an increasing homeomorphism $\eta:\R_+\to \R_+$ and a constant $C>0$ such that 
\beq
\label{defqs}
\frac{\left|\Phi(z)-\Phi(z_0)\right|}{|\Phi(w)-\Phi(z_0)|}\le C \, \eta\left(\frac{|z-z_0|}{|w-z_0|}\right)  
\eeq
for all $z_0 \in \RN$ and $z, w\in \R^N\setminus\{z_0\}$.
Set, for $\alpha>0$ and $t \ge 0$,
\beq
\label{defetaH}
\eta_\alpha(t)=\max\bigl\{t^\alpha, t^{1/\alpha}\bigr\}.
\eeq
By \cite[Theorems 3.18 and 5.1]{AVV},  any $K$-quasiconformal map $\Phi$ fulfils \eqref{defqs} with 
\[
C=C(K), \qquad \eta=\eta_{K^{1/(N-1)}}.
\]
The {\em maximal linear dilatation} of  $\Phi$ is defined as the (finite) number
\[
H=\esssup_{z\in \R^N} \frac{\sigma_{\rm max}(D\Phi(z))}{\sigma_{\rm min}(D\Phi(z))}.
\]
Elementary linear algebra shows that such a quasiconformal $\Phi$ is actually $H^{N-1}$-quasiconformal, hence it fulfils the distortion estimate
	 \beq
	 \label{quasisymmetry}
\frac{\left|\Phi(z)-\Phi(z_0)\right|}{|\Phi(w)-\Phi(z_0)|}\le C\, \eta_{H}\left(\frac{|z-z_0|}{|w-z_0|}\right)
\eeq
for all $z_0\in \R^N$ and all $z, w\in \R^N\setminus\{z_0\}$.
Notice that the constant $C$ actually depends on $N$ as well. Finally, recall that if  $\Phi$ is quasiconformal then so is $\Phi^{-1}$ and the maximal linear dilatations of $\Phi$ and $\Phi^{-1}$ coincide. Hence \eqref{quasisymmetry} holds true for $\Phi^{-1}$ as well.

\begin{definition}
A map $\Phi:\R^N\to \R^N$ is {\em $\delta$-monotone} for some $\delta\in \, ]0, 1]$ if 
\[
\left(\Phi(z)-\Phi(w), z-w\right)\ge \delta\,  |\Phi(z)-\Phi(w)|\, |z-w|.
\]
\end{definition}

Kovalev's theorem \cite{K} shows that any non-constant $\delta$-monotone map is quasiconformal. In particular it is a homeomorphism and, directly from the definition, its inverse is $\delta$-monotone as well.
Kovalev theorem has a quantitative version proved in \cite[Theorem 1]{AIM}. It states that any $\delta$-monotone map has a maximal linear dilatation obeying
\[
H\le \frac{1+\sqrt{1-\delta^2}}{1-\sqrt{1-\delta^2}}
\]
and this bound is sharp, i.\,e.\,it reduces to an equality for the $\delta$-monotone linear map $v\mapsto A\, v$, where 
\beq
\label{defA}
A=
\begin{pmatrix}
1+\sqrt{1-\delta^2}&0\\
0&1-\sqrt{1-\delta^2}
\end{pmatrix}.
\eeq
The opposite implication is in general not true. To see this, consider the linear map $v\mapsto B\, v$, with $B$ given by
\[
B=
\begin{pmatrix}
1&-2\\
2&-1
\end{pmatrix}
\]
which is quasiconformal with $H=3$ but   not even monotone. However, full (even quantitatively) equivalence of the two concepts holds true in the  class of {\em gradient mappings}.

\begin{definition}
A differentiable function  $F:\R^N\to \R$ is  $H$-quasiuniformly convex (briefly, $H$-q.\,u.\,c.) if it is convex and its gradient map is quasiconformal, with maximal linear dilatation bounded by $H$.
\end{definition}
 More explicitly (and by making use of Aleksandrov's theorem), $F$ is $H$-q.\,u.\,c.\,if
 \begin{itemize}
 \item[(i)]
 $F$ is convex, $C^1$ and $W^{2, 1}_{\rm loc}(\R^N)$, and not affine;
 \item[(ii)]
 It holds
\beq
\label{quc}
\lambda_{\rm max}(D^2 F(z))\le H\, \lambda_{\rm min}(D^2F(z))
\eeq
 for a.\,e.\,point $z$ of second order differentiability.
 \end{itemize}
 
By \cite[Theorem 3.1 and Lemma 3.2]{KoMa}, any q.\,u.\,c.\,function  is strictly convex and coercive, thus it has a unique minimum point.
Since, given two symmetric matrices $M_1, M_2$,  it holds
\beq
\label{ex3}
\lambda_{\rm min}(M_1+M_2)\ge \lambda_{\rm min}(M_1)+\lambda_{\rm min}(M_2), \qquad \lambda_{\rm max}(M_1+M_2)\le \lambda_{\rm max}(M_1)+\lambda_{\rm max}(M_2),
\eeq
 the set of $H$-q.\,u.\,c.\,functions turns out to be a convex cone, i.\,e.\,it is closed by sum and positive scalar multiples. Moreover, it is also closed by isometric change of variables, as well as dilations. Finally, summing an affine function to a $H$-q.\,u.\,c.\,function still gives a $H$-q.\,u.\,c.\,function.  For these reasons we will frequently {\em normalize} q.\,u.\,c.\,integrands by requiring that
 \beq
 \label{hypF}
 F(z)\ge F(0)=0
 \eeq
 and that $i_F=1$, where
 \beq
 \label{defif}
 i_F:=\inf_{|z|=1} |DF(z)|,
 \eeq
 simply by considering
 \beq
 \label{tildeF}
 \tilde{F}(z)=\frac{1}{i_F}\, \left(F(z+\bar z)-F(\bar z)\right), \qquad \text{where } {\rm Argmin}\, (F)=\{\bar z\}.
 \eeq
 
\begin{ex}
\label{exsd}
Some examples of q.\,u.\,c.\,integrands have already been discussed in the previous section. Here, given two points $z_1\ne z_2$ in $\R^N$, we construct a q.\,u.\,c.\,function $F$ such that ${\rm Sing}_F=\{z_1\}$ and ${\rm Deg}_F=\{z_2\}$, the latter sets being as in \eqref{sing-deg}. 

Without loss of generality, assume $z_1=0$ and $z_2=w$, for fixed $w\in \R^N$, and set $r=|w|$. For $p>2>q>1$ choose
\[
d(z)=\frac{|z|^{p}}{p},  \qquad s(z)= \frac{|z-w|^{q}}{q},
\]
so that a direct computation shows (cf. \cite[Example 3.6]{GM})
\begin{align}
\lambda_{\rm min}(D^2d(z))&= |z|^{p-2},  & \lambda_{\rm max}(D^2d(z))&=(p-1)\, |z|^{p-2}, \label{ex1}\\
\lambda_{\rm min}(D^2s(z))&= (q-1)\, |z|^{q-2}, & \lambda_{\rm max}(D^2 s(z))&= |z|^{q-2}. \label{ex2}
\end{align}
Choose $\varphi\in C^\infty(\R^N\setminus \{w\})$  such that 
\[
\varphi(z)=
\begin{cases}
s(z)&\text{if $|z-w|\le r/4$}\\
1 &\text{if $|z-w|\ge r/2$}
\end{cases}
\]
and consider the (positive and finite) numbers
\begin{align*}
 \alpha&=\inf\left\{\lambda_{\rm min} (D^2d(z)):|z|\ge r/2\right\},  &\beta&=\sup \left\{|D^2\varphi(z)|_2:|z-w|\ge r/4\right\},\\
 \gamma&=\sup\left\{\lambda_{\rm max} (D^2d(z)):|z-w|\le r/4\right\},  &\delta&=\inf\left\{\lambda_{\rm min}(D^2 s(z)):|z-w|\ge r/4\right\}.
  \end{align*}
For any $\eps>0$, the function $F(z)=d(z)+\eps\, \varphi(z)$ belongs to  $C^1(\R^N)\cap W^{2,N}_{\rm loc}(\R^N)$ and furthermore $F\in C^2(\R^N\setminus\{w\})$, so that ${\rm Sing}_F\subseteq \{w\}$. 

For  $z\in B_{r/2}(0)$ we have 
\[
\lambda_{\rm min}(D^2F(z))=\lambda_{\rm min}(D^2d(z)), \qquad \lambda_{\rm max}(D^2(F(z))=\lambda_{\rm max}(D^2d(z)),
\]
which shows through \eqref{ex1} that 
\[
\frac{\lambda_{\rm max}(D^2F(z))}{\lambda_{\rm min}(D^2F(z))}\le p-1\qquad \text{in $B_{r/2}(0)$}
\]
and $ \lambda_{\rm min}(D^2F(z))\to 0$ for $z\to 0$.
From \eqref{ex3} we infer that for all $z\in B_{r/4}(w)$
\[
\lambda_{\rm min}(D^2F(z)) \ge \alpha+  \eps \, \lambda_{{\rm min}}(D^2s(z)), \qquad  \lambda_{\rm max}(D^2F(z)) \le \gamma+ \eps\, \lambda_{{\rm max}}(D^2s(z)),
\]
so that by \eqref{ex2}  it follows
\[
\frac{\lambda_{\rm max}(D^2F(z))}{\lambda_{\rm min}(D^2F(z))}\le \max\l\{\frac{\gamma}{\alpha}, \frac{1}{q-1}\r\}  \qquad \text{in $B_{r/4}(w)$}
\]
and $\lambda_{\rm min}(D^2F(z))\to \infty$ for $z\to w$.
Finally, for  $z\notin B_{r/2}(0)\cup B_{r/4}(w) $  we have
\[
\lambda_{\rm min}(D^2F(z))\ge \lambda_{\rm min}(D^2d(z)) -\eps\, \beta, \qquad \lambda_{\rm max}(D^2F(z))\le  \lambda_{\rm max}(D^2d(z)) +\eps\, \beta,
\]
so that for $\eps\, \beta<\alpha/2$ it holds $\lambda_{\rm min}(D^2F(z))\ge \alpha/2$, while  \eqref{ex1} yields 
\[
\frac{\lambda_{\rm max}(D^2F(z))}{\lambda_{\rm min}(D^2F(z))}\le 3\, (p-1)\qquad \text{in $\R^N\setminus\big(B_{r/2}(0)\cup B_{r/4}(w)\big)$}.
\]
All in all, for $\eps<\alpha/(2\, \beta)$, $F$ is q.\,u.\,c.\,with ${\rm Sing}_F=\{w\}$, ${\rm Deg}_F=\{0\}$.
\end{ex}

Many properties of q.\,u.\,c.\,functions have been studied in  \cite{K, KoMa}, to which we refer for further details and characterisations. Here we gather the ones that are needed in the proof of our main result.
We start by a converse of Kovalev's theorem relating the $\delta$-monotonicity to the quasiconformality in a quantitative form.

\begin{lemma}\label{lemma25}
Let  $F:\R^N\to \R$ be  differentiable and not affine. Then, $F$ is $H$-q.\,u.\,c.\,if and only if $DF$ is $\delta$-monotone, where 
\[
\delta=\frac{2\, \sqrt{H}}{H+1}, \qquad \text{or} \qquad H=\frac{1+\sqrt{1-\delta^2}}{1-\sqrt{1-\delta^2}},
\]
and the bounds are sharp.
 \end{lemma}

\begin{proof}
The fact that if $DF$ is $\delta$-monotone and non-constant then $F$ is  $H$-q.\,u.\,c.   has been proved, as mentioned before, in \cite[Theorem 1]{AIM}. To prove the opposite implication, suppose $F$ is $H$-q.\,u.\,c., so that $DF$ is quasiconformal. By \cite[Section 3]{K}, the $\delta$-monotonicity of $DF$  is equivalent to
\beq
\label{matin}
\left(D^2F(z)\, v, v\right)\ge \delta\, |D^2F(z)\, v|\, |v|
\eeq
at a.\,e.\,points of second order differentiability where \eqref{quc} holds true. Note that we can assume that $D^2F$ is symmetric, as this follows from Alexandrov's theorem (see \cite[Corollary 2.9]{R}), and that it is strictly positive definite, thanks to the quasiconformality of $DF$. Since \eqref{matin} is invariant by orthogonal change of variables, we can assume that $D^2F$ is diagonal with positive eigenvalues $\lambda_1<\lambda_2< \dots< \lambda_N$. We are therefore reduced to find, for $v=(v_1, \dots, v_N)$, the value of
\[
I:=\inf_{v\ne 0} \frac{\sum_{i}\lambda_i\, v_i^2}{\Big({\sum_i \lambda_i^2\, v_i^2}\Big)^{1/2} \Big({\sum_i v_i^2}\Big)^{1/2}}.
\]
In order to determine the latter we invoke Cassels inequality (see \cite[Appendix]{W}), which reads as
\[
\frac{\sum_i w_i\, a^2_i\sum_iw_i\, b_i^2}{\left(\sum_i w_i\, a_i\, b_i\right)^2}\le \frac{(M+m)^2}{4\, M\, m}
\]
for any choice of $w_i\ge 0$ not all identically equal to zero and $a_i, b_i>0$ such that
\[
0<m<a_i/b_i<M.
\]
Indeed, it suffices to choose $b_i=1$, $a_i=\lambda_i$ and $w_i=v_i^2$ to obtain that 
\[
I\ge \frac{2\, \sqrt{\lambda_1\, \lambda_N}}{\lambda_1+\lambda_N}
\]
and actually the equality holds true for the vector
\[
 \left(\sqrt{\frac{\lambda_N}{\lambda_1+\lambda_N}}, 0, \dots, 0, \sqrt{\frac{\lambda_1}{\lambda_1+\lambda_N}}\right).
\]
Since $\lambda_N\le H\, \lambda_1$ by assumption and the map $t\mapsto 2\sqrt{t}/(t+1)$ is decreasing, we infer that
\[
I\ge  \frac{2\, \sqrt{\lambda_1\, \lambda_N}}{\lambda_1+\lambda_N}= \frac{2\, \sqrt{\lambda_N/ \lambda_1}}{\lambda_N/\lambda_1+1}\ge \frac{2\, \sqrt{H}}{H+1}
 \]
 so that \eqref{matin} is proved for $\delta=2\, \sqrt{H}/(H+1)$, as claimed. 
 The optimality of this estimate follows by choosing $F(z)=(A\, z, z)$ with $A$ given in \eqref{defA} (here we use the fact that $A$ is symmetric).
 \end{proof}

For further reference, we regroup the previous discussion and some other useful properties of $H$-q.\,u.\,c.\,functions in the following proposition.

\begin{proposition} \label{proqu} 
Let $F:\R^N\to \R$ be a $H$-q.\,u.\,c. function. The following holds:
\begin{enumerate}
\item
$F$ is strictly convex, coercive and $W^{2, N}_{\rm loc}(\R^N)$.
\item
Both $DF$ and $DF^{-1}$ are $\eta_H$-quasisymmetric and  $2\sqrt{H}/(H+1)$-monotone.
\item
There exists $C=C(H, N)>0$ such that, if $F(\bar z)=\min_{\R^N} F$, then
\begin{align}
\label{deltamon2}
& |DF(z)-DF(\bar z)|\, |z-\bar z|\le C\, |F(z)-F(\bar z)| \\
\label{estF}
 & \frac{F(w)-F(\bar z)}{F(z)-F(\bar z)}\le C\, \frac{|w-\bar z|}{|z-\bar z|}\, \eta_H\left( \frac{|w-\bar z|}{|z-\bar z|}\right)
\end{align}
for all $z, w\in \R^N\setminus\{\bar z\}$.
\item
For any $\varphi\in C^\infty_c(\R^N; [0, +\infty))$   the function $F*\varphi (z)+\mu\, |z|^2$ is $H$-q.\,u.\,c. Moreover, if $\|\varphi\|_1=1$, for all sequences $(\eps_n)$, $(\mu_n)$ such that $\eps_n\downarrow 0$, $\mu_n\downarrow 0$, the sequence
\beq
\label{appr0}
F_n(z)=F*\varphi_{\eps_n}(z)+\frac{\mu_n}{2}\, |z|^2, \qquad \text{with } \varphi_{\eps_n}(z)=\frac{1}{\eps_n^N}\, \varphi\left(\frac{z}{\eps_n}\right),
\eeq
is such that $F_n \to F$ in $C^{1}_{\rm loc}(\R^N)$ and $DF_n^{-1}\to DF^{-1}$ in $C^0_{\rm loc}(\R^N)$.
\item
For any $\delta>0$, the Moreau-Yoshida reguarization $F_\delta$ of $F$, defined as
\[
F_{\delta}(z)= \inf_{w\in \R^{N}}\Big\{F(w)+\frac{1}{2\, \delta} |w-z|^{2}\Big\},
\]
is $H$-q.\,u.\,c. and it holds $\displaystyle \lambda_{\rm max}(D^2F_\delta(z)) \le 1/\delta$. 
%if $P_\delta=\big({\rm Id}+\delta\, DF^{-1}\big)^{-1}$, for a.\,e.\,$z\in \R^N$ it holds
%\beq
%\label{MY2}
%\begin{split}
%\lambda_{\rm min}(D^2F_\delta(z))&=\frac{\lambda_{\rm min}\big(D^2F(P_\delta(z))\big)}{1+\delta\, \lambda_{\rm min}\big(D^2F(P_\delta(z))\big)}, \\
 %\lambda_{\rm max}(D^2F_\delta(z))&=\frac{\lambda_{\rm max}\big(D^2F(P_\delta(z))\big)}{1+\delta\, \lambda_{\rm max}\big(D^2F(P_\delta(z))\big)}.
%\end{split}
%\eeq
\end{enumerate}
\end{proposition}

\begin{proof}

%From \cite{MG} we see that $DF\in W^{1, N}_{\rm loc}(\R^N)$, therefore $DF$ is quasiregular and then open, continuous and discrete by  Reshetniak's theorem.  If $DF(z_0)=DF(z_1)=\xi$ then
%\[
%F(z)\ge F(z_0)+(\xi, z-z_0), \qquad  F(z)\ge F(z_1)+(\xi, z-z_1)
%\]
%so that for any $z\in \R^N$ and $t\in [0, 1]$
%\[
%\begin{split}
%F(z)&=t\, F(z)+(1-t)\, F(z)\ge t\, \left(F(z_0)+(\xi, z-z_0)\right)+(1-t)\, \left(F(z_1)+ (\xi, z-z_1)\right)\\
%&\ge F(t\, z_0+(1-t)\, z_1)+\left(\xi, z-(t\, z_0+(1-t)\, z_1)\right).
%\end{split}
%\]
%Therefore, being $F\in C^1$, it holds $\xi=DF(t\, z_0+(1-t)\, z_1)$, proving that  $DF^{-1}(\xi)$ is convex for any $\xi\in\R^N$.
%Since the latter is discrete, it must be a point, proving that $DF$ is injective and thus that $F$ is strictly convex by \cite[Corollary 26.3.1]{R}. Since $DF$ is open, it  is a homeomorphism on its image, hence quasiconformal. Finally, by \cite{V}, $DF$ is surjective and thus a homeomorphism of $\R^N$. The last assertion follows from strict convexity, since the minimum point is characterised by $\bar z=DF^{-1}(0)$. 
%
%To prove {\em (2)},  from \cite{V}, it follows that $DF$ is $H^{N-1}$-quasiconformal in the geometric sense, therefore so is $DF^{-1}$. The regularity of the latter follows from \cite{V}. The $\delta$ monotonicity of $DF$ is proved in \cite{H}, and immediately implies $\delta$-monotonicity of $DF^{-1}$ as well.

The first two properties have already been discussed, so we focus on  {\em (3)}. We will prove it by exploiting the $\delta$-monotonicity and the quasisymmetry of $DF$. To this aim we note that, by considering $F(z+\bar z)-F(\bar z)$, we can assume $\bar z=0$ and $F(z)\ge F(0)=0$.
To prove \eqref{deltamon2}, we start by noticing that the $\delta$-monotonicity of $DF$ gives
\beq
\label{F1}
F(z)=\int_0^1 \left(DF(t\, z), z), z\right)\, dt\ge \delta\, |z|\, \int_{1/2}^1 |DF(t\, z)|\, dt.
\eeq
On the other hand,  inequality \eqref{quasisymmetry} applied to $DF$ with $z_0=0$ (and since $DF(0)= 0$, too) gives
\[
\frac{|DF(z)|}{|DF(t\, z)|}\le C\,  \eta_H(1/t)\le C\,  2^H\qquad \forall \, t\ge 1/2,
\]
so that  \eqref{F1} can be estimated as
\[
F(z)\ge \frac{\delta}{ C\, 2^{H+1}} \, |DF(z)|\, |z|.
\]
Inequality \eqref{estF} is proven similarly. It holds
\[
F(w) \le \int_0^1 |DF( t\, w)|\, |w|\, dt\le |w |\, \sup\left\{ |DF(x)|:|x |\le |w |\right\}
\]
which combined with \eqref{deltamon2} yields
\[
\frac{F(w)}{F(z)}\le C\, \frac{|w|}{|z|}\, \sup\left\{ \frac{|DF(x)|}{|DF(z)|}:|x |\le |w |\right\}.
\]
Thanks to the quasisymmetry \eqref{quasisymmetry} of $DF$ we get
\[
\frac{|DF(x)|}{|DF(z)|}\le C\, \eta_H\left(\frac{|x|}{|z|}\right)\le C\, \eta_H\left(\frac{|w|}{|z|}\right)
\]
whenever $|x|\le |w|$, and thus \eqref{estF} follows.

The first part of assertion {\em (4)}  has been proved in \cite[Proposition 2.3]{GM}, while the $C^1_{\rm loc}(\R^N)$ convergence of $F_n$ to $F$ is trivial. Since $DF_n$ are homeomorphisms which are locally uniformly converging to $DF$, the same is true for $DF^{-1}$ thanks to Arens' theorem (see \cite{D}).
Finally, part {(5)} has been proved in \cite[Proposition 2.3-(iv)]{GM}, where it in particular it is shown that if $P_\delta=\big({\rm Id}+\delta\, DF^{-1}\big)^{-1}$, for a.\,e.\,$z\in \R^N$ it holds
\[
\begin{split}
\lambda_{\rm min}(D^2F_\delta(z))&=\frac{\lambda_{\rm min}\big(D^2F(P_\delta(z))\big)}{1+\delta\, \lambda_{\rm min}\big(D^2F(P_\delta(z))\big)}, \\
 \lambda_{\rm max}(D^2F_\delta(z))&=\frac{\lambda_{\rm max}\big(D^2F(P_\delta(z))\big)}{1+\delta\, \lambda_{\rm max}\big(D^2F(P_\delta(z))\big)}.
\end{split}
\]
Since 
\[
\sup_{t\ge 0} \frac{t}{1+\delta\, t}=\frac{1}{\delta},
\]
the claimed bound follows.
\end{proof}

Recall that by a smooth {\em strongly elliptic integrand} we mean an $F\in C^\infty(\R^N)$ such that 
\[
\lambda_{\rm min}(D^2F(z))\ge \lambda, \qquad \lambda_{\rm max}(D^2F(z))\le \Lambda
\]
for some $0<\lambda\le \Lambda<\infty$. Clearly, strongly elliptic integrands are $H$-q.\,u.\,c.\,with $H=\Lambda/\lambda$.

\begin{corollary}\label{cclosure}
Let $H\ge 1$. The cone
\[
{\rm C}_H=\{F:\R^N\to \R \text{ s.\,t.\,$F$ is $H$-q.\,u.\,c.\,or affine}\}
\] 
is closed with respect to point-wise a.\,e.\,convergence and the smooth, strongly elliptic integrands in ${\rm C}_H$ are dense in it. Moreover, point-wise a.\,e.\,convergence of $(F_n)_n$ in ${\rm C}_H$ implies $C^1_{\rm loc}(\R^N)$ convergence of $(F_n)_n$ and $L^1_{\rm loc}(\R^N)$ convergence of $({\rm det}\, D^2F_n)_n$.
\end{corollary}

\begin{proof}
By \cite[Lemma 2.5]{KoMa} we know that the point-wise limit of q.\,u.\,c.\,functions is either q.\,u.\,c.\,or affine and that the point-wise a.\,e.\,convergence implies the $C^1_{\rm loc}(\R^N)$ one. Therefore it suffices to show that 
\[
F\mapsto \esssup_{z\in \R^N} \frac{\lambda_{\rm max}(D^2F(z))}{\lambda_{\rm min}(D^2F(z))}
\]
is lower semicontinuous with respect to $C^1_{\rm loc}(\R^N)$ convergence on the cone of q.\,u.\,c.\,functions.  Let $F_n\in {\rm C}_H$ be such that $F_n\to F$ in $C^1_{\rm loc}(\R^N)$ with  none of the $F_n$ affine. By Lemma \ref{lemma25}, $DF_n$ is $2\, \sqrt{H}/(H+1)$-monotone hence, by passing to the limit in the definition od $\delta$-monotonicity, so is $DF$. By Lemma \ref{lemma25} again, this implies that $F$ is $H$-q.\,u.\,c., giving the claimed lower semicontinuity. 

To prove the density statement, let $F$ be $H$-q.\,u.\,c.\,for some $H$. For $\delta_n, \eps_n, \mu_n \downarrow 0$, the functions 
\[
F_n(z)=F_{\delta_n}*\varphi_{\eps_n}(z)+\frac{\mu_n}{2}|z|^2
\]
(here $F_{\delta_n}$ denotes the Moreau-Yoshida regularisation of $F$) constructed through Proposition \ref{proqu}, points {\em (4)} and {\em (5)}, are $H$-q.\,u.\,c.\ integrands that approximate $F$ point-wise. Using \eqref{ex3} together with basic properties of convolution and Proposition \ref{proqu}-{\em (5)}, yields  
\[
\lambda_{\rm max}(D^2F_n(z))\le \frac{1}{\delta_n}+\mu_n,\qquad \lambda_{\rm min}(D^2F_n(z))\ge \mu_n
\]
for all $z\in \R^N$, hence $F_n$ are strongly elliptic integrands. 

If instead $F$ is an affine function of the form $F(z)=(w, z)+c$  for fixed $w\in \R^N$ and $c\in \R$, we can consider $G(z)=|z|^2/2$ and set
\[
F_n(z)=n\big(G(w+z/n)-G(w)\big)+c, \qquad DF_n(z)=DG(w+z/n), \qquad D^2F_n=\frac{{\rm Id}}{n}.
\]
Clearly $F_n\to \ell$ in $C^1_{\rm loc}(\R^N)$ and  the $F_n$ are $1$-q.\,u.\,c.\,and strongly elliptic. The last stated property on the $L^1$-convergence of the determinants is contained in \cite[Lemma 2.5]{KoMa}.
\end{proof}

Actually (see again \cite[Lemma 2.5]{KoMa}), the pointwise convergence in the previous corollary can be weakened to hold on a dense subset of $\R^N$.

\begin{remark}
The previous corollary, as outlined in its proof, implies the lower semicontinuity of the maximal linear dilatation of monotone, quasiconformal gradient maps. Lower semicontinuity of the maximal linear dilatation in the larger class of quasiconformal maps holds true for $N=2$, but fails for $N\ge 3$, due to a famous counterexample of Iwaniec \cite{I1}. 
\end{remark}

We conclude this section by introducing a generalisation of the function $\eta_H$ given in \eqref{defetaH}, which will appear in many subsequent computations. Given $a, b>0$, the auxiliary functions 
\[
\eta_{a, b}(t)=\max\bigl\{t^a, t^b\bigr\}, \qquad t\ge 0,
\]
 are increasing homeomorphisms of $\R_+$ to itself, and it holds $\eta_H=\eta_{H, 1/H}$.
Since  
\[
%\label{maxmin}
\max\bigl\{t^a, t^b\bigr\}=s\quad \Longleftrightarrow\quad t=\min\bigl\{s^{1/a}, s^{1/b}\bigr\}, \qquad \forall \, a, b, t, s>0,
\]
we see that 
\beq
\label{eta-inv}
\eta_{a, b}^{-1}(t)=\min\bigl\{t^{1/a}, t^{1/b}\bigr\}.
\eeq
We collect in the following proposition some elementary properties of the function $\eta_{a, b}$ and of its inverse.
\begin{proposition}
\label{eta-lemma}
Let $a, b>0$. Then,
\begin{enumerate}
\item
For all $s, t\ge 0$
\beq
\label{basiceta}
\eta_{a, b} (s\, t)\le \eta_{a, b}(s)\, \eta_{a, b}(t) \quad \text{as well as} \quad 
\eta_{a, b}^{-1}(s\, t)\ge \eta_{a, b}^{-1}(s)\, \eta_{a, b}^{-1}(t)
\eeq
\item
For all $t>0$
\beq
\label{eta01}
\eta_{a, b}^{-1}(t)\, \eta_{1/a, 1/b}(1/t)= 1.
\eeq
\item
For all $t>0$ and $\sigma>0$ it holds
	\beq
	\label{dis-eta}
	\frac{1}{C}\,  \eta_{a, b}(t) \le \eta_{a, b}(\sigma\, t) \le C\,  \eta_{a, b}(t),
	\eeq
for positive constants $C=C(a, b, \sigma)$. The same estimate holds true for $\eta_{a, b}^{-1}$ as well.
\item
For $a> b>0$ and $ c>d>0$
\beq
\label{compo}
\eta_{a, b}\circ \eta_{c, d}=\eta_{ac , bd}
\eeq
and the same formula holds true for the inverses.
\end{enumerate}
\end{proposition}

The proof of these facts is elementary and is therefore omitted.

 \section{Local minimisers and Sobolev regularity of their stress field}
 \label{3}
\begin{definition}
\label{def:local-min}

Let $\Om \subset\RN$ be open. For all $\mathcal O \Subset \Om$ let us consider the functionals
  $J\colon W^{1,1} ({\mathcal O}) \to \R$ given by
	\beq
	\label{J}
	J(w, {\mathcal O})= \int_{\mathcal O} F(Dw) \; dx.
	\eeq
We say that a function $u \in W^{1,1}_{\rm loc}(\Om)$ is a local minimiser of $J$ {in $\Om$}  if $F(Du)\in L^1_{\rm loc}(\Omega)$ and for any open ${\mathcal O} \Subset\Om$ it holds
	\[
	J(u, {\mathcal O})= \inf\l\{J(w, {\mathcal O}): \, w \in u+ W^{1,1}_0({\mathcal O})\r\}.
	\]
\end{definition}
If $F$ is $C^1$, strictly convex and coercive, let $\{\bar z\}={\rm Argmin}\, (F)$. It is readily checked that any local minimiser $u$ of $J$ defines a local minimiser $\tilde{u}(x)=u(x)-(\bar z, x)$ for the functional $\tilde{J}$ given as in \eqref{J}, but with $\tilde{F}(z)$ defined as in \eqref{tildeF}. In particular, we can assume that \eqref{hypF} holds true, i.\,e.\,that $\bar z=0$. If $F$ is q.\,u.\,c., which will be assumed henceforth, using \eqref{estF} for $\bar z=0$ gives
\[
F(z)\ge \frac{1}{C}\,  F(w)\, |z|\, \eta^{-1}_H(|z|)\qquad \forall \, |w|=1,
\]
so that for all $|z|>1$ it holds
\begin{equation}
\label{cc}
F(z)\ge\frac{1}{C} \sup_{|w|=1} F(w) \, |z|^{1+\frac{1}{H}}.
\end{equation}
In particular, any local minimiser for $J$ in $\Omega$ belongs to $W^{1, 1+\frac{1}{H}}_{\rm loc}(\Omega)$.

More precisely, despite the fact that q.\,u.\,c.\,integrands are {\em per se} anisotropic, they enjoy, together with their derivative, an isotropic growth control allowing to treat minimisation problems for $J$ in standard Orlicz-Sobolev spaces. Indeed, after normalising $F$ so that \eqref{hypF} holds true, we can let for $t\ge 0$
  \[
%  \label{defPhi}
 a(t)=\sup_{|z|\le t} |DF(z)|, \qquad   A(t)=\int_0^ta(\tau)\, d\tau.
  \]
  Notice that $a$ is continuous on $[0, \infty[$ and positive for $t>0$. Given $0\le t_1<t_2$, since $DF$ is a homeomorphism of $\R^N$, $DF(B_{t_2})$ is an open neighbourhood of the closed set $DF(\overline{B_{t_1}})$, thus $a(t_2)>a(t_1)$, i.\,e.\,$a$ is strictly increasing. 
   By the quasisymmetry of $DF$, choosing $z_0=0$ and $|w|=t$ in  \eqref{quasisymmetry}, we obtain that for any $z\in \R^N$  it holds
  \[
  |DF(z)|\le C\, a(t) \, \eta_H(|z|/t), 
  \]
  so that for any $k\ge 0$, $t\ge 0$, the monotonicity of $\eta_H$ grants  
  \beq
  \label{ak}
  a(k\, t)\le C\, \eta_H(k)\, a(t).
  \eeq
    Therefore $A$ is a Young function, i.\,e. it is convex, increasing and fulfils
   \[
   \lim_{t\to 0^+ }\frac{A(t)}{t}=0, \qquad \lim_{t\to \infty} \frac{A(t)}{t}=\infty,
   \]
where the second equality follows from \eqref{ak} for $k=1/t$ and using $\eta_H(1/t)^{-1}=t^{1/H}$ for $t\ge 1$.
The inverse of $a$ is readily checked to be
  \[
  a^{-1}(s)=\inf_{|z|\ge s} |DF^{-1}(z)|
  \] 
  which obeys a similar estimate as \eqref{ak}.
It follows that $A$ satisfies the $\Delta_2$ condition
\[
A(2\, t)\le C\, A(t)\qquad \forall \, t\ge 0
\]
for $C=C(H, N)$, and the same holds true for its Young conjugate defined as
\[
A^*(s)=\sup_{t\ge 0} \left(s\, t -A(t)\right)=\int_0^s a^{-1}(\sigma)\, d\sigma
\]
 (see the proof of \cite[eq. (16.94)-(16.95)]{AIMbook}). 

Clearly
  \[
  F(z)= \int_0^{|z|} (DF(s\, z/|z|), z/|z|)\, ds\le \int_0^{|z|} a(s)\, ds=A(|z|)
  \]
  while the $\delta$-monotonicity and quasisymmetry of $DF$ give
  \[
  F(z)\ge \delta\, \int_0^{|z|} |DF(s\, z/|z|)|\, ds\ge \frac{\delta}{C}\, \int_0^{|z|} a(s)\, ds,
  \]
  so that
    \beq
\label{cPhi}
\frac{1}{C}\, A(|z|)\le F(z)\le C\, A(|z|).
\eeq
The map $DF$ enjoys similar isotropic bounds, namely,
\[
\frac{1}{C}\, a(|z|)\le |DF(z)|\le a(|z|),
\]
where the first inequality follows from \eqref{quasisymmetry} as before, and it is furthermore possible to show that
    \beq
\label{cDPhi}
\frac{1}{C}\, A(|z|)\le A^*(DF(z))\le C\, A(|z|)
\eeq
(see the proof of \cite[eq. (16.106)]{AIMbook}).

Given an open set ${\mathcal O}\Subset\Omega$, it follows from \eqref{cPhi} that 
\[
J(u, {\mathcal O})<\infty\quad \text{if and only if}\quad A(|Du|)\in L^1({\mathcal O}).
\]
Using the fact that ${\mathcal O}$ has finite measure and $A$ satisfies the $\Delta_2$ condition, the summability of $A(|Du|)$ can be equivalently stated as 
\[
u\in W^{1, A}({\mathcal O}):=\left\{v\in W^{1,1}({\mathcal O}): A(|v|), A(|Dv|)\in L^1({\mathcal O})\right\},
\]
equipped with the Luxembourg norm
\[
\|v\|_{L^A({\mathcal O})}+\||Dv|\|_{L^A({\mathcal O})}, \qquad \|g\|_{L^A({\mathcal O})}=\inf\left\{s>0:\int_{\mathcal O}A(|g|/s)\, dx\le 1\right\}.
\]
Thanks to the $\Delta_2$ condition on $A$, the so-called Orlicz-Sobolev space $W^{1,A}({\mathcal O})$ so defined turns out to be a Banach space. 
 Existence of local minimiser given, say, a boundary datum $\varphi \in W^{1, A}(\Omega)$ for $\Omega$ bounded is thus granted by minimising $J$  on $u+W^{1,A}_0(\Omega)$, the latter space being the closed vector space of  those $v\in W^{1, A}(\Omega)$ whose extension at zero outside $\Omega$ belongs to $W^{1,1}(\R^N)$. Indeed, thanks to the validity of the $\Delta_2$ condition on both $A$ and $A^*$, the Banach space $W^{1,A}_0(\Omega)$ is reflexive by \cite[p. 54]{DT}, while $J$ is convex and coercive by \eqref{cPhi}. Conversely, a local minimiser $u$ on $\Omega$ minimises $J(\cdot \,, {\mathcal O})$ on $u+W^{1, A}_0({\mathcal O})$  for any ${\mathcal O}\Subset \Omega$ and thus standard methods (see 
\cite[Theorem 2.1]{DG}) ensure the validity of the Euler-Lagrange equation
 \beq
 \label{weakDF}
 \int_{\mathcal O} (DF(Du), Dw)\, dx= 0 \qquad \forall \, w\in W^{1,A}_0({\mathcal O}).
 \eeq
Notice that for any $u\in W^{1, A}({\mathcal O})$ the inequalities in \eqref{cDPhi} imply that $|DF(Du)|\in L^{A^*}({\mathcal O})$ (thus {\em a fortiori} $|DF(Du)|\in L^1({\mathcal O})$) and therefore  H\"older's inequality in Orlicz spaces ensures that 
  \[
  W^{1,A}_0({\mathcal O})\ni w\mapsto \int_{\mathcal O}(DF(Du), Dw)\, dx
  \]
  is a well defined continuous linear functional. Thus, if for ${\mathcal O}\Subset\Omega$ a function $u\in W^{1,A}({\mathcal O})$ fulfils
  \[
  {\rm div}\, (DF(Du))=0
  \]
  in the distributional sense, it does so also in the weak sense \eqref{weakDF}, thanks to the density of  $C^\infty_0({\mathcal O})$ in $W^{1, A}_0({\mathcal O})$ granted  by \cite[Theorem 2.1]{DT} (here   the $\Delta_2$ condition on $A$ alone is sufficient).
  
  We summarise the previous discussion in the following proposition.

\begin{proposition}\label{minimi}
Let $N\ge 2$ and $F$ be a q.\,u.\,c.\,integrand. A function $u\in W^{1,1}_{\rm loc}(\Omega)$ is a local minimiser for $J$ in \eqref{J} if and only if $F(Du)\in L^1_{\rm loc}(\Omega)$ and 
\[
{\rm div}\, (DF(Du))=0
\]
in the distributional sense.
\end{proposition}

The next proposition is essentially contained in \cite[Theorem 3.3]{GM}, but we will need the more explicit estimates given in its present form. It revolves around the Sobolev regularity of the {\em stress field} 	
\beq
	\label{def-V}
	V(x)= DF(Du(x)).
	\eeq

\begin{proposition}\label{prosob}
Let $F$ be a $H$-q.\,u.\,c.\,function obeying \eqref{hypF} and let $i_F$ be as in \eqref{defif}.
Let $u$ be a local minimiser for $J$, and let  $V$ be given by \eqref{def-V}. Then, $V\in W^{1,2}_{\rm loc}(\Omega)$ with estimates
\begin{subequations}
\label{proso}
\begin{align}
\left(\intmed_{B_R}|DV|_2^2\, dx\right)^{\frac{1}{2}}&\le \frac{C}{R} \, i_F\, \eta_{\frac{H}{H+1}, \frac{1}{H+1}}\left(\frac{1}{i_F}\intmed_{B_{2R}} F(Du)\, dx\right), \label{proso1}\\
\left(\intmed_{B_R}|V|^2\, dx\right)^{\frac{1}{2}}&\le C\, i_F\, \eta_{\frac{H}{H+1}, \frac{1}{H+1}}\left(\frac{1}{i_F}\intmed_{B_{2R}} F(Du)\, dx\right), \label{proso2}
\end{align}
\end{subequations}
for a positive constant $C=C(N, H)$ and all balls $B_R$ such that $B_{2R}\subseteq \Omega$. 
%Moreover, with the same notations,
%\beq
%\label{proso2}
%\eta_{\frac{H}{H+1}, \frac{1}{H+1}}\left(\frac{1}{i_F}\intmed_{B_{2R}} F(Du)\, dx\right)\le C\,\|Du\|_{L^{F/i_F}(B_{2R})}.
%\eeq
\end{proposition}

\begin{proof}
 From \cite[Theorem 3.3, (3.10)]{GM} we know that
 \[
 \|V\|_{W^{1,2}(B_R)}\le C \|V\|_{L^1(B_{2R})}
 \]
 for a constant $C=C(N, H, R)$. More precisely, setting
 \beq
 \label{mr}
 m_{r}=\intmed_{B_{r}} |V|\, dx,
 \eeq
  it holds (see \cite[proof of Theorem 3.2]{GM} {or Corollary \ref{dato} below})
\beq
\label{re001}
\| D V\|_{L^2(B_R)}\le C(N, H)\, R^{\frac{N}{2}-1} \, m_{2R}
\eeq
for a constant $C$ which henceforth depends only on $H$ and $N$. On the other hand, by Poincar\'e inequality and this last estimate we have
 \beq
 \label{re002}
 \begin{split}
 \|V\|_{L^2(B_{R})}&\le \|V-m_{R}\|_{L^2(B_{R})}+ (\omega_N\, R^N)^{\frac{1}{2}}m_{R}\\
 &\le C\, R\, \|DV\|_{L^2(B_{R})} +C\, R^{\frac{N}{2}}\, m_{2R}\\
 &\le C\, R^{\frac{N}{2}}\, m_{2R}.
 \end{split}
 \eeq
We claim that
\beq
\label{re0}
m_{2R} \le C\,  i_F\, \eta_{\frac{H}{H+1}, \frac{1}{H+1}}\left(\frac{1}{i_F}\intmed_{B_{2R}} F(Du)\, dx\right)
\eeq
for a constant $C=C(N, H)$.
To this aim, for any $t>0$ we proceed as follows
  \beq
  \label{re}
  \begin{split}
  \int_{B_{2R}} |DF(Du)|\, dx&=\int_{\{|Du|\le t\}\cap B_{2R}} |DF(Du)|\, dx+\int_{\{|Du|> t\}\cap B_{2R}} |DF(Du)|\, dx\\
  & \le C\, \sup_{|z|\le t} |DF(z)|\, R^N+\frac{1}{t} \int_{B_{2R}}|DF(Du)|\, |Du|\, dx.
  \end{split}
  \eeq
The first term on the above right-hand side is estimated by using the quasisymmetry of $DF$ \eqref{defqs} as
\[
|DF(z)|\le C\, \inf_{|w|=1} |DF(w)|\, \eta_H(|z|),
\]
while on the second term we use \eqref{deltamon2}. Therefore \eqref{re} reduces to
\[
\int_{B_{2R}} |DF(Du)|\, dx\le C\, R^N\, i_F\,  \left(  \eta_H(t)+\frac{1}{t} \, \intmed_{B_{2R}} \frac{F(Du)}{i_F}\, dx\right).
\]
Choose $t$ such that 
\[
t\,  \eta_H(t)=  \intmed_{B_{2R}} \frac{F(Du)}{i_F}\, dx\quad \Longleftrightarrow \quad t=\eta_{1+\frac{1}{H}, 1+H}^{-1}\left(\intmed_{B_{2R}} \frac{F(Du)}{i_F}\, dx\right)
 \]
 to get through \eqref{eta01}
 \[
 \int_{B_{2R}} |DF(Du)|\, dx\le C\, R^N\, i_F\,  \eta_{\frac{1}{H+1}, \frac{H}{H+1}}\left(\left(\intmed_{B_{2R}} \frac{F(Du)}{i_F}\, dx\right)^{-1}\right) \intmed_{B_{2R}} \frac{F(Du)}{i_F}\, dx.\\
 \]
 Finally, notice that
 \[
t\,  \eta_{\frac{1}{H+1}, \frac{H}{H+1}}(1/t)= \eta_{\frac{1}{H+1}, \frac{H}{H+1}}(t)
\]
for all $t\ge 0$, so that taking \eqref{mr} and \eqref{def-V} into account implies the validity of  \eqref{re0}.
Recalling  \eqref{re001} and \eqref{re002} finally gives \eqref{proso1} and \eqref{proso2}, respectively.
\end{proof}

 \section{A regularising procedure}
\label{4}
 As customary, it will be convenient to work under the assumption that both the integrand and the minimiser under scrutiny are smooth. In order to do so, we introduce a family of regularised problems having the required smoothness and whose integrands and corresponding minimisers  converge to the original problem.

 \begin{lemma}[Regularised problems]\label{regularized}
 Let $F$ be a $H$-q.\,u.\,c.\,function fulfilling \eqref{hypF} and let $u$ be a local minimiser for $J$ in $W^{1,1}(\Omega)$. Then, for any ball $B_{2R}\Subset \Omega$ there exists a sequence $F_n\in C^\infty(\R^N)$ of $H$-q.\,u.\,c. integrands obeying \eqref{hypF} and 
 \[
 \inf_{z\in \R^N} \lambda_{\rm min} (D^2F_n(z))>0
 \]
  and a sequence of corresponding minimisers $u_n\in C^{\infty}(B_{2R})$  of 
 \[
 J_n(w, B_{2R})=\int_{B_{2R}} F_n(Dw)\, dx
 \]
such that
 \begin{subequations}
  \label{prosobe}
\begin{align}
 \int_{B_{2R}} F_n(Du_n)\, dx &\to \int_{B_{2R}}F(Du)\, dx,  \label{prosobe1}\\
 u_n &\rightharpoonup u \quad \text{in $W^{1, 1+1/H}(B_R)$},  \label{prosobe2} \\
 DF_n(Du_n) &\rightharpoonup DF(Du) \quad \text{in $W^{1,2}(B_R)$}.  \label{prosobe3}
\end{align}
 \end{subequations}
 \end{lemma}

\begin{proof}
 
We set for brevity $B=B_{2R}$ and $p=1+1/H$. Fix an even $\varphi\in C^\infty_c(B_1; [0, +\infty[)$ such that $\|\varphi\|_1=1$ and use the notation of  Proposition \ref{proqu}-(4). Notice that for any such $\varphi$ it holds
\[
z=\int (z-y)\, \varphi(y)\, dy
\]
so that Jensen inequality gives 
\beq
\label{J0}
F(z)\le (\varphi \ast F)(z).
\eeq 
Set 
\[
F_n(z)=(\varphi_{\eps_n}*F) (z)+\frac{\mu_n}{2}\, |z|^2,
\]
with $\eps_n\downarrow 0$ and $\mu_n\downarrow 0$ to be chosen, so that $F_n$ fulfils for each $n$ all the conditions stated in the lemma.
 Notice that thanks to \eqref{cc} and the convergence $F_n \to F$ in $C^1_{\text{loc}}(\RN)$ we have
\begin{equation}
\label{equicc}
F_n(w)\ge \frac{|w|^{p}}{C} \quad \text{as well as} \quad \frac{1}{C} \le \inf_{|z|=1} |DF_n(z)|\le C,
\end{equation}
respectively, for constants $C$ independent of $n$.

Let, for sufficiently large $n$,
\[
\psi_n=u*\varphi_{1/n}\in C^\infty(B)
\]  
and consider the minimisation problem
\begin{equation}
\label{probn}
\inf\{J_{n}(w):w\in {\rm Lip}(\overline{B}), w=\psi_n \ \text{on $\partial B$}\}.
\end{equation}
 According to \cite[Theorem 9.2]{S}, there is a solution $u_{n} $ of \eqref{probn}, which (see \cite[p. 5923]{BMT} and \eqref{cc}) also solves 
\[
J_{n}(u_{n})=\inf\{J_{n}(w):w\in \psi_n+W^{1,p}_{0}(B)\}.
\]
By  \cite[Theorem 4.1]{BB} there are constants $A_{n}\ge 1$ (depending only on $B$ and on the regularity of $\psi_n$, but not on $\eps_{n}$, $\mu_{n}$) such that 
\[
{\rm Lip}(u_{n})\le \frac{A_{n}}{\mu_{n}}.
\]
The integrand $F$ is therefore strongly elliptic on the range of $Du_n$ since $F_n\in C^\infty(\R^N)$ and 
\[
\lambda_{\rm min}(D^2F_n(Du_n))\ge \mu_n, \qquad \lambda_{\rm max}(D^2F_n(Du_n))\le C<\infty
\]
on $\overline{B}$, so that  standard regularity theory gives $u_n\in C^\infty(\overline{B})$.
 We first choose $\mu_{n}\downarrow 0$ so that 
\beq
\label{condmun}
\lim_{n}\mu_{n}\int_{B}|D\psi_n|^{2}\, dx=0.
\eeq
Set
\[
M_{n}=1+\sup_{B} |D\psi_n| + \frac{A_n}{\mu_n}
\]
and observe that, since $\varphi_{\eps_n} * F\to F$ in $C^1_{\rm loc}(\R^N)$ as $\eps_n\downarrow 0$ 
and $M_{n}$ is independent of $\eps_{n}$, we can pick $(\eps_{n})\subseteq (0, 1)$, $\eps_{n}\downarrow 0$, so that 
\beq
\label{epsn}
\|\varphi_{\eps_n}*F-F\|_{C^1(B_{M_n})}\le \frac{1}{n}.
\eeq
Clearly, it still holds $F_{n}\to F$ in $C^{1}_{{\rm loc}}(\R^{N})$.
 We claim that \eqref{prosobe} holds true with such choices.

By the minimality of $u_{n}$, it holds $J_n(u_n)\le J_n(\psi_n)$, while from   \eqref{condmun} we get
\beq
\label{varlimsup}
\varlimsup_{n} J_{n}(u_n)\le \varlimsup_{n} J_{n}(\psi_n)= \varlimsup_{n} \int_{B} (\varphi_{\eps_{n}}*F)(D\psi_n)\, dx.
\eeq
In order to estimate the rightmost term of \eqref{varlimsup} we use  \eqref{epsn}  to have
\beq
\label{varph}
 \int_{B} (\varphi_{\eps_{n}}*F) (D\psi_n)\, dx\le \frac{|B|}{n}+\int_{B}F(D\psi_n)\, dx,
 \eeq
while the vector-valued Jensen inequality  implies
\[
\begin{split}
\int_{B}F(D\psi_n)\, dx&=\int_{B}F\left(\int \varphi_{1/n}(x-y)\, Du(y)\, dy\right)\, dx\\
&\le \int_{B}\left(\int \varphi_{1/n}(x-y) F(Du(y))\, dy\right) dx\\
&\le \int_{B}\varphi_{1/n}*F(Du)\, dx.
\end{split}
\]
Therefore from \eqref{varph} we have
\[
 \varlimsup_{n} \int_{B} (\varphi_{\eps_{n}}*F)(D\psi_n)\, dx\le \varlimsup_{n}\int_{B}\varphi_{1/n}*F(Du)\, dx=\int_B F(Du)\, dx
 \]
which, inserted into  \eqref{varlimsup}, gives
 \beq
 \label{limsup}
 \varlimsup_{n} J_{n}(u_{n})\le J(u).
 \eeq
In particular, $J_n(u_n)$ is bounded and \eqref{equicc} then implies a uniform bound on $Du_n$ in $L^{p}(B)$. 
   Moreover, since $u_n-\psi_n\in W^{1, p}_0(B)$, then  Poincar\'e's inequality gives
 \[
 \begin{split}
 \|u_{n}\|_{L^{p}(B)}&\le \|u_{n}-\psi_n\|_{L^{p}(B)}+\|\psi_n\|_{L^{p}(B)}\\
 &\le C\, \big(\|D(u_{n}-\psi_n)\|_{L^{p}(B)}+\|\psi_n\|_{L^{p}(B)}\big)\\
 &\le C\, \big(\|Du_{n}\|_{L^{p}(B)}+\|\psi_n\|_{W^{1,p}(B)}\big),
 \end{split}
 \]
 so that $(u_{n})$ is bounded in $L^{p}(B)$ as well. Therefore  $(u_{n})$   possesses a (not relabeled) subsequence weakly converging in $W^{1,p}(B)$ to some $v$  and  it is readily checked that $v\in u+W^{1,p}_{0}(B)$.  This in turn implies that \eqref{prosobe2} is satisfied.

Now, since  the map
\[
w\mapsto \int_{B}F(Dw)\, dx
\]
is weakly lower semicontinuous in $W^{1,p}(B)$ and $F_n\ge F$ thanks to \eqref{J0}, we get
\beq
\label{qpl}
J(v)\le \varliminf_{n} \int_{B} F(Du_{n})\, dx\le\varliminf_{n} \int_{B} F_n(Du_{n})\, dx = \varliminf_{n} J_{n}(u_{n}).
\eeq
 Coupling the latter with \eqref{limsup} gives  $ J(v)\le J(u)$, implying $v=u$ by the strict convexity of $F$. 
 In particular,
 % the first part of \eqref{prosobe} follows and, 
 up to subsequences,
 \beq
 \label{wvn}
 Du_{n}\rightharpoonup Du \qquad \text{in $L^{p}(B)$}.
 \eeq
Therefore, from \eqref{qpl} and  \eqref{limsup} we infer 
   \beq
 \label{fn3}
  \int_B F(Du_n)\, dx\to \int_B F(Du)\, dx,
  \eeq
that is \eqref{prosobe1}. We now  apply Proposition \ref{prosob} to $ DF_n(Du_n)$, thus obtaining 
\[
\|D F_{n}(Du_{n})\|_{W^{1, 2}(B_R)}\le C,
\]
where the uniform bound holds thanks to equations \eqref{equicc} and \eqref{fn3}, being $B=B_{2R}$.
Setting
 \beq
 \label{Vn}
V_{n}=DF_{n}(Du_{n})
\eeq
  we can pick a subsequence such that $V_{n}\to V$ weakly in $W^{1,2}(B_R)$, strongly in $L^{2}(B_R)$, and pointwise a.\,e.\,in $B_R$,
for a suitable $V\in W^{1,2}(B_R)$. Since   $DF_n^{-1}\to DF^{-1}$ locally uniformly thanks to Proposition \ref{proqu}-(4), we infer that $DF_n^{-1}(V_n) \to DF^{-1}(V)$ a.\,e.. Then \eqref{Vn} implies that $Du_n \to DF^{-1}(V)$  a.\,e.,  and in turn \eqref{wvn} allows the identification $V=DF(Du)$. This shows \eqref{prosobe3} and the proof is thus complete.
   \end{proof}

 \section{Caccioppoli inequality}
\label{5}
In this section we prove a Caccioppoli-type inequality. To this aim, let us first show the following result, which is a simplification of \cite[Lemma 3.1]{GM}.

\begin{lemma}\label{lemma51}
Let $P$ and $S$ be $N\times N$ symmetric matrices, with  $P$ positive definite. Then,
\beq
\label{matr}
\left(P\, S, S\, P\right)_2\ge \frac{\lambda_{\rm min}(P)}{\lambda_{\rm max}(P)} |P\, S|_2^2.
\eeq
\end{lemma}

\begin{proof}
Let us observe that both sides of \eqref{matr} are invariant by orthogonal change of basis, so we can assume that $P$ is diagonal, with positive  eigenvalues $\lambda_1, \dots, \lambda_N$ obeying
\[
\lambda_j\ge \frac{1}{\beta} \, \lambda_i, \quad \text{where }  \beta=\frac{\lambda_{\rm max}(P)}{\lambda_{\rm min}(P)},
\]
for any $i, j=1, \dots, N$.
Then, if $S=(a_{ij})$, we have
\[
\left(P\, S, S\, P\right)_2=\sum_{i, j=1}^N\lambda_i\,  a_{ij} \, \lambda_j\, a_{ji}=\sum_{i, j=1}^N\lambda_i\, \lambda_j\, a_{ij}^2\ge \frac{1}{\beta}\sum_{i, j=1}^N\lambda_i^2\, a_{ij}^2=\frac{1}{\beta}\, |P\, S|_2^2,
\]
as claimed.
\end{proof}

\begin{lemma}[Caccioppoli Inequality]\label{caclemma}
Let $F$ be a smooth $H$-q.\,u.\,c.\,function and let $u\in C^\infty(B_R)$.
% and set $V(x)=DF(Du(x))$. 
 For $V$ given by \eqref{def-V}, assume that 
\beq
\label{div}
{\rm div}\, V=0\qquad \text{in $B_R$}.
\eeq
Moreover, for  a given affine function  $\ell:\R^N\to \R$,   and all $\rho<R$, set 
\beq
\label{A-kr}
A(\ell, \rho)=\left\{x\in B_\rho: F(Du(x))\ge  \ell(Du(x))\right\}.
\eeq
Then, for   all $\rho<R$, it holds
	\beq
	\label{cac}
	\int_{A( \ell, \rho)}|DV|_2^2\, dx\le \left(\frac{\pi\,  H }{R-\rho}\right)^2\,  \int_{A( \ell, R)} |V-D\ell|^2\, dx.
	\eeq
	\end{lemma}

	\begin{proof}	
Let $\eps>0$ and set $V_\eps=V+\eps\, Du=DF(Du)+\eps\, Du$. Differentiating the latter gives
 \[
 DV_\eps=\big(D^2F(Du)+\eps\, {\rm Id}\big)\, D^2 u,
 \]
where the matrices $D^2F(Du)+\eps\, {\rm Id}$ and $D^2u$ are  symmetric  positive definite and symmetric, respectively. By \eqref{quc}, we have
\[
\begin{split}
\lambda_{\rm max}(D^2F(Du)+\eps\, {\rm Id})& \le \lambda_{\rm max}(D^2F(Du))+\eps \\
&\le H\, \lambda_{\rm min}(D^2F(Du))+\eps\\
&\le H\, \lambda_{\rm min}(D^2F(Du)+\eps\, {\rm Id})
\end{split}
\]
hence  \eqref{matr}  gives $|DV_\eps|_2^2\le H \, (DV_\eps,  DV_\eps^t)_2$.
Letting $\eps\downarrow 0$   provides
\[
	|DV|_2^2\le H \, (DV,  DV^t)_2.
\]
We want to integrate the above inequality  over $A( \ell, r)$,  for $r< R$, assuming for the moment that $\{F(Du)=\ell(Du)\}$ is a $C^1$ hypersurface in $B_R$. 
Using the divergence theorem gives
	\beq
	\label{eq2}
	\begin{split}
	\int_{A( \ell, r)}|DV|_2^2\, dx&\le H\, \int_{A( \ell, r)} \left(DV, DV^t\right)_2\, dx\\
	&= H\, \int_{A( \ell, r)}\sum_{i, j=1}^N V^i_j\, (V^j-\ell_j)_i\, dx\\
	&= H\, \int_{A( \ell, r)}\sum_{i, j=1}^N \Big(V^i_j\, (V^j-\ell_j)\Big)_i-\sum_{j=1}^N (V^j-\ell_j)\, \Div V_j\, dx\\
	&= H\,  \int_{\partial A(\ell, r)} \sum_{i, j=1}^NV^i_j \, (V^j-\ell_j)\, n^i\, d{\mathcal H}^{N-1},
	\end{split}
	\eeq
where $n$ denotes the exterior normal to $\partial A( \ell, r)$ and where we used the fact that \eqref{div} implies $\Div V_j=0$ for all $j= 1, \dots, N$. 
To compute the last integral in \eqref{eq2}, notice that 
	\[
	 \partial A( \ell, r)= \bigl(\partial B_r\cap \{F(Du)\ge  \ell(Du)\}\bigr) \bigcup \,  \bigl(B_r \cap \{F(Du)=\ell(Du)\}\bigr),
	\]
so that 
	\[
	n(x)=
	\begin{cases} 
	\dfrac{x}{|x|}&\text{if $x\in  \partial B_r\cap \{F(Du)\ge  \ell(Du)\}$}\\[15pt]
	-\dfrac{D\big((F-\ell)(Du(x))\big)}{|D\big((F-\ell)(Du(x))\big)| }&\text{if $x\in B_r \cap \{F(Du)=\ell(Du)\}$}.
	\end{cases}
	\]
Let $x \in B_r \cap \{F(Du)=\ell(Du)\}$ and compute
\[
U^i(x):=\big((F-\ell)(Du(x)) \big)_i =\sum_{k=1}^N(V^k-\ell_k)\, u_{ki} \quad \text{as well as} \quad  V^i_j=\sum_{h=1}^N F_{i h}(Du) \, u_{h j},
\]
 so that
	\[
	\begin{split}
	\sum_{i, j=1}^NV^i_j \, (V^j-\ell_j)\, n^i&=-\frac{1}{|U|} \sum_{i, j, k=1}^NV^i_j\, (V^j-\ell_j)\, (V^k-\ell_k)\, u_{k i}\\
	&=- \frac{1}{|U|}\,\sum_{i,j,k,h=1}^N F_{i h}(Du)\,  u_{h j}\, (V^j-\ell_j)\, u_{i k}\, (V^k-\ell_k)\\
	&= -\frac{1}{|U|}\, \big(D^2F(Du)\, U, U\big)< 0,
	\end{split}
	\]
	being $D^2F$non-negative definite.
Then the previous inequality implies that the rightmost integral in \eqref{eq2} evaluated on $B_r \cap \{F(Du)=\ell(Du)\}$ has a negative sign and hence \eqref{eq2} reduces to
 	\beq
	\label{eq3}
	\begin{split}
	\int_{A( \ell, r)}|DV|_2^2\, dx&\le H\,  \int_{ \partial B_r\cap \{F(Du)\ge  \ell(Du)\}}\sum_{i, j=1}^N V^i_j\, (V^j-\ell_j)\, n^i\, d{\mathcal H}^{N-1}\\
	&\le H\, \int_{ \partial B_r\cap \{F(Du)\ge  \ell(Du)\}} |DV|_2\, |V-D\ell|\, d{\mathcal H}^{N-1},
	\end{split}
	\eeq
where in the last step we used the  Schwarz inequality. 

 Let now $\zeta$ be the function given by
\[
	\zeta(r)=\int_{ A( \ell, r)} |DV|_2\, |V-D\ell|\, dx \quad \text{for all } 0<r<R,
	\]
with first derivative 
	\beq
	\label{h-primo}
	 \zeta'(r)= \int_{  \partial B_r\cap \{F(Du)\ge  \ell(Du)\}}|DV|_2\,  |V-D\ell|\, d{\mathcal H}^{N-1}.
	\eeq
Let  $\varphi\in C^\infty([0, R])$ with $\varphi(R)=0$.
We  multiply \eqref{eq3} by $\varphi^2$ and integrate in $[0, R]$. Taking  \eqref{h-primo} into account, integrating by parts and using the fact that $\varphi(R)= \zeta(0)=0$, we get
\[
	\begin{split}
	 \int_0^R\hspace{-3pt}\int_{A( \ell, r)}\varphi^2\, |DV|_2^2\, dx\, dr &\le H\, \int_0^R\varphi^2\, \zeta'\, dr\\
	 & =H\left(\varphi^2(R)\,\zeta(R)-\varphi^2(0)\, \zeta(0)\right)  - 2\, H\, \int_0^R\varphi\, \varphi'\, \zeta \, dr \\
	&\le 2\, H\,  \int_0^R\hspace{-3pt} \int_{A( \ell, r)} |\varphi| |\varphi'| |DV|_2 |V-D\ell| \; dx\,  dr \\
	& \le 2\, H \l({\int_0^R\hspace{-3pt}\int_{A( \ell, r)}\varphi^2\, |DV|_2^2\, dx dr}\r)^{\frac 1 2}\hspace{-3pt} \l({\int_0^R\hspace{-3pt}\int_{A( \ell, r)}|\varphi'|^2\, |V-D\ell|^2\, dx dr}\r)^{\frac 1 2}
	\end{split}
	\]
which simplifies to
	\[
	\int_0^R\hspace{-3pt}\int_{A( \ell, r)}\varphi^2\, |DV|_2^2\, dx\, dr\le 4\, H^2\, \int_0^R\hspace{-3pt}\int_{A( \ell, r)}|\varphi'|^2\, |V-D\ell|^2\, dx\, dr.
	\]
Fix $0< \rho< R$. Taking advantage of the monotonicity of the functions
\[
r\mapsto  \int_{A( \ell, r)}|DV|_2^2\, dx, \qquad r\mapsto  \int_{A( \ell, r)}|V-D\ell |_2^2\, dx
\]
 we then have
	\[
%	\label{eq4}
	\begin{split}
	  \int_{A( \ell, \rho)}|DV|_2^2\, dx\, \int_\rho^R\varphi^2\, dr&\le \int_\rho^{R}\hspace{-3pt}\int_{A( \ell, r)} \varphi^2\, |DV|_2^2\, dx\, dr\\
	  &\le \int_0^{R}\hspace{-3pt}\int_{A( \ell, r)} \varphi^2\, |DV|_2^2\, dx\, dr\\
	&\le 4\, H^2\, \int_0^R\hspace{-3pt}\int_{A( \ell, r)}|\varphi'|^2\, |V-D\ell|^2\, dx\, dr\\
	 	&\le   4\, H^2  \int_{A( \ell, R)} |V-D\ell|^2\, dx\, \int_0^R|\varphi'|^2\, dr,
	\end{split}
	\]
	which implies
		\[
	\int_{A( \ell, \rho)}|DV|_2^2\, dx\le 4\, H^2\, I(R, \rho)\,  \int_{A( \ell, R)} |V-D\ell|^2\, dx 
	\]
	for
	\[
	I(R, \rho):=\inf \left\{\dfrac{\displaystyle{\int_0^R|\varphi'|^2\, dr}}{\displaystyle{\int_\rho^R\varphi^2\, dr}}: \varphi\in C^\infty([0, R])\setminus \{0\}, \varphi(R)=0\right\}.
	\]
	A minimiser for this problem is found through elementary considerations as
	\[
	\varphi(r)=
	\begin{cases}
	1&\text{if $r\le \rho$}\\[5pt]
	\cos\left(\dfrac{\pi}{2}\, \dfrac{r-\rho}{R-\rho}\right) &\text{if $r\in [\rho, R]$},
	\end{cases}
	\]
	so that an explicit computation gives
	\[
	I(R, \rho)=\frac{\pi^2}{4\, (R-\rho)^2},
	\]
	and \eqref{cac} follows.
	
 To remove the assumption that the set $\{F(Du)=\ell(Du)\}$ is $C^1$ in $B_R$, {we invoke Sard's theorem}, 
 which ensures that this property is satisfied by 
$\{F(Du)=\ell(Du)+k\}$, for a.\,e.\,$k\in \R$. Given $k_n\downarrow 0$, we can repeat the same argument as before, starting from \eqref{eq2}, thus arriving to
	\begin{equation}
	\label{cac-1}
	\int_{A(\ell+ k_n, r)} |DV|_2^2 \; dx \le \l(\frac{\pi \, H}{R- \rho}\r)^2 \int_{A(\ell+ k_n, R)} |V- D\ell|^2 \; dx.
	\end{equation}
Notice that $A(\ell+k_n, \cdot)\searrow A(\ell, \cdot)$,  so we can pass to the limit in \eqref{cac-1} by monotone (or dominated) convergence, and in turn  \eqref{cac} holds true also in this case.
\end{proof}

We report the following consequence, already present in \cite[proof of Theorem 3.2]{GM}, for the sake of completeness.

\begin{corollary}
\label{dato}
Let $F$ be a $H$-q.\,u.\,c. function, let $u$ be a local minimiser for $J$ in $\Omega$ in the sense of Definition \ref{def:local-min} and let $V$ be given by \eqref{def-V}. Then,
\beq
\label{cac1}
\int_{B_R} |DV|_2^2\, dx\le \frac{C}{R^{N+2}}\left(\int_{B_{2R}} |V|\, dx\right)^2
\eeq
for $C=C(H, N)$ and all balls $B_R$ such that $B_{4R}\subseteq\Omega$.
\end{corollary}

\begin{proof}
By using Lemma \ref{regularized} on $B_{2R}$ we can suppose that both $F$ and $u$ are smooth, since for the regularised problems the left-hand side of \eqref{cac1} is lower semicontinuous with respect to the weak convergence in $W^{1,2}(B_R)$, while the right-hand side is (up to subsequences) convergent. We thus apply  Lemma \ref{caclemma} for $\ell(z)\equiv \min F$ to get
\beq
\label{ja}
\int_{B_\rho} |DV|_2^2\, dx\le \frac{\widetilde C}{(r-\rho)^2}\, \int_{B_{r}} |V|^2\, dx
\eeq
for all $0<\rho<r<2R$. Let $E$ be a continuous extension operator
\[
E: W^{1, 2}(B_1; \R^N)\to W^{1,2}_0(B_2; \R^N).
\]
The Gagliardo-Nirenberg inequality \cite[Theorem 12.83]{Leoni} ensures that 
\[
\|E(U)\|_{2}\le C\, \|E(U)\|_{1}^{\frac{2}{N+2}}\, \|DE(U)\|_{2}^{1-\frac{2}{N+2}} \qquad \forall \, U\in W^{1,2}(B_1; \R^N)
\]
 for $C=C(N)$, so that by the continuity of $E$
 \[
 \begin{split}
\int_{B_{1}} |U|^2\, dx&\le \int_{B_2} |E(U)|^2\, dx\le C\, \left(\int_{B_2} |E(U)|\, dx\right)^{{\frac{4}{N+2}}}\left(\int_{B_2}|DE(U)|^2\, dx\right)^{{\frac{N}{N+2}}}\\
& \le C\, \left(\int_{B_1} |U|\, dx\right)^{\frac{4}{N+2}}\left(\int_{B_1}|DU|^2\, dx\right)^{{\frac{N}{N+2}}}
\end{split}
\]
for a bigger constant depending on the operator norms of $E:L^1(B_1; \R^N)\to L^1(B_2; \R^N)$ and $E:W^{1,2}(B_1; \R^N)\to W^{1,2}_0(B_2; \R^N)$.
For any $\eps>0$ we apply to the previous estimate Young's inequality with exponents $(N+2)/2$ and $(N+2)/N$ to get 
\[
\int_{B_{1}} |U|^2\, dx\le \eps^2\, \int_{B_1} |DU|^2\, dx+\frac{C}{\eps^N}\left(\int_{B_1}|U|\, dx\right)^2,
\]
 which rescales to 
\beq
\label{sobolev}
\int_{B_{r}} |U|^2\, dx\le (\eps\, r)^2\,  \int_{B_{r}} |DU|^2\, dx+\frac{C}{(\eps\, r)^N}\left(\int_{B_{r}}|U|\, dx\right)^2 \qquad \forall \, U\in W^{1,2}(B_r; \R^N)
\eeq
for some $C=C(N)>0$.
Choose
\[
\eps^2=\frac{1}{2} \frac{(r-\rho)^2}{\widetilde C\, r^2}
\]
%with $\widetilde C$ given in \eqref{ja}  
and apply \eqref{sobolev} with $U=V$ and such $\eps$  to \eqref{ja} to get
\[
\int_{B_\rho} |DV|_2^2\, dx\le \frac{1}{2}\,  \int_{B_r} |DV|_2^2\, dx+ \frac{C}{(r-\rho)^{N+2}}\left(\int_{B_{r}}|V|\, dx\right)^2
\]
for all $0<\rho<r<2R$. An application of \cite[Lemma 3.1, Ch. 5]{Gia} shows that then
\[
\int_{B_\rho} |DV|_2^2\, dx\le  \frac{C}{(r-\rho)^{N+2}}\left(\int_{B_{r}}|V|\, dx\right)^2
\]
which, for $r=2R$, $\rho=R$, provides \eqref{cac1}.
\end{proof}

\section{A family of $1$-homogeneous functions}
\label{6}

Let $F$ be a $H$-q.\,u.\,c.\,function.  In this section we gather the basic properties of the function 
	\beq
	\label{G}
	G(z)=F(DF^{-1}(z))
	\eeq
 and of the family of Minkowski functionals related to it, namely, 
 	\beq
	\label{minkowski}
	g_k(z)= \inf\{t>0 :\, G(z/t)< k\},
	\eeq
for any $k \ge 0$. Note that, since $DF$ is a homeomorphism of $\RN$, the function in \eqref{G} is well defined.

	\begin{proposition}
	Let $F$ be a $H$-q.\,u.\,c.\,function  fulfilling \eqref{hypF}.
	 Then, $G$   is coercive and
	for any $k\ge 0$ the set $\{G\le k\}$ is  star-shaped with respect to the origin.
	\end{proposition}

	\begin{proof}
We apply \eqref{deltamon2} with the vector $DF^{-1}(z)$ and use \eqref{hypF} to have
	\[
	G(z)\ge  |DF^{-1}(z)|\, |z|/C.
	\]
Furthermore,  the quasisymmetry \eqref{quasisymmetry} of $DF^{-1}$ together with $DF^{-1}(0)=0$ give
	\[
	|DF^{-1}(z)|\ge \frac{\sup_{|v|=1}|DF^{-1}(v)|}{C\, \eta_H(1/|z|)}.
\]
Gathering the previous inequalities and making use of \eqref{eta01} and \eqref{eta-inv} we  have
\[
G(z)\ge\frac{\sup_{|v|=1}|DF^{-1}(v)|}{C}\,    |z|^{1+1/H}  \quad \text{for all } |z|> 1,
\]
thus proving the coercivity of $G$. In particular for any $ k>0$ the set $\{G\le  k\}$
 is compact with nonempty interior.  
 
 Let us now show that  $\{G\le k\}$ is star shaped with respect to the origin. Since $G(0)=0$ by the normalisation \eqref{hypF}, it suffices to show that for any $z\in \S^{N-1}$ the function $t \mapsto G(t\, z)$ is nondecreasing.  To this end, suppose first that $F\in C^2(\R^N)$ fulfils
 \beq
 \label{sc}
 \lambda_{\rm min}( D^2F(z))\ge \mu
 \eeq
  for some $\mu>0$ independent of $z$. Then $DF$ is a diffeomorphism and it holds
	\beq
	\label{DG}
	DG(z)=z\,  D(DF^{-1})(z)=(D^2F)^{-1}(DF^{-1}(z))\, z,
	\eeq
therefore 
\[
	\frac{d}{dt} G(t\, z)= \left(DG(t\, z), z\right)=  t  \left( (D^2F)^{-1}(DF^{-1}(t\, z))\, z, z\right)	\ge  0,
\]
proving the claim. In the general case, thanks to Proposition \ref{proqu}-{(4)}, we know that any $H$-q.\,u.\,c.\,function $F$ can be approximated in $C^1$ by a sequence $(F_n)$ of $C^2$ and $H$-q.\,u.\,c. functions obeying \eqref{sc}, and such that $DF_n^{-1} \to DF^{-1}$ locally uniformly. Hence, $G_n=F_n\circ DF_n^{-1}\to G$ locally uniformly, too. It follows that $t\mapsto G(t\, z)$ is nondecreasing, being the pointwise limit of a sequence of nondecreasing functions. The proof is thus complete.
 \end{proof}

\begin{remark}
In general, the sets $\{w\in \R^N: G(w)\le k\}=\{DF(z): F(z)\le k\}$ are {\em not} convex, so that the functions $g_k$ defined in \eqref{minkowski} may fail to be  norms. Moreover, by employing the approximation in \eqref{appr0} in the previous proof, we see that the $H$-q.\,u.\,convexity assumption is not needed to prove the second statement of the previous proposition and, given any convex $F\in C^1(\R^N)$ obeying \eqref{hypF}, the sets $\{ DF(z): F(z)\le k\}$, while not convex in general, are always star-shaped with respect to the origin.
\end{remark}

\begin{proposition}
Let $F$ be a $C^2$ and $H$-q.\,u.\,c.\,function obeying \eqref{hypF} and 
\[
\inf_{z\in \R^N} \lambda_{\rm min}(D^2 F(z))>0.
\]
%Let  $G=F\circ DF^{-1}$ and for any $k>0$
%\[
%g_k(z)=\inf\left\{t>0: G(z/t)< k\right\}.
%\]
Then, the following holds:
\begin{enumerate}
\item
$g_k$ is a Lipschitz $1$-homogeneous function with
\beq
\label{Dgkappa}
{\rm Lip}\, (g_k)\le H\, \sup_{\S^{N-1}} g_k.
\eeq
\item
For some $C=C(H, N)>0$ it holds
\beq
\label{supinfg}
\sup_{\S^{N-1}} g_k\le C\, \inf_{\S^{N-1}} g_k.
\eeq
\item
There exists $C=C(H, N)>0$ such that for any $k\ge h>0$ 
\beq
\label{claim1}
\inf_{\{G\ge k\}}(g_h-1)\ge \frac{1}{C}\, \min\l\{ \left(\frac{k-h}{h}\right)^H, \left(\frac{k-h}{h}\right)^{\frac{1}{H(H+1)}}\r\}.
\eeq
\end{enumerate}
\end{proposition}

\begin{proof}
Since $DG$ given in \eqref{DG} doesn't vanish outside the origin,  the implicit function theorem implies that the boundary $\partial \{G\le k\}=\{G=k\}$ is a $C^1$-hypersurface
and therefore $g_k$ is $C^1$ outside the origin. By construction $g_k$ is positively $1$-homogeneous, hence it fulfils
	\begin{align}
	\label{prog1}
	\big(Dg_ k(z), z\big)&=g_ k(z), \quad  \forall \, z\ne 0, \\
	\label{prog2}
	 Dg_ k(\lambda\, z)&=Dg_ k(z), \quad \forall \, \lambda>0, z\ne 0, 
	 \end{align}
	 while by construction it holds
	 \beq
	 \label{prog3}
	 G\left(\frac{z}{g_ k(z)}\right)=k \qquad \forall \, z\ne 0.
	\eeq
From the last property, in particular, it holds $g_ k\equiv 1$ on $ \{G= k\}$,  and therefore the Lagrange multiplier rule gives $Dg_ k(z)=\alpha(z)\, DG(z)$ for some $\alpha(z)\ge 0$ and all $z\in \{G= k\}$. 
\color{black}
This, together with  \eqref{prog1}, implies that on  $\{G= k\}$ we have
	\[
	1=g_ k(z)=(Dg_ k(z), z)= \alpha(z)\, (DG(z), z),
	\]
and in turn 
	 \[
	 Dg_ k(z)=\frac{DG(z)}{\big(DG(z), z\big)}.
	 \]
Let $z\in \{G= k\}$. Thanks to \eqref{DG} we get 
	 \beq
	 \label{dgk}
	 |Dg_ k(z)|=\frac{|DG(z)|}{\big(DG(z), z\big)}= \frac{|z \, (D^2F)^{-1}(w)|}{\big(z\, (D^2F)^{-1}(w), z\big)}, \qquad \text{with } w=DF^{-1}(z).
	 \eeq
 Let $\Lambda(w)$ and $\lambda(w)$ be the maximum and minimum  eigenvalues
  of $D^2F$ at $w$, respectively. Then it holds
	 \beq
	 \label{eigenv}
	 |z\, (D^2F)^{-1}(w)|\le \frac{1}{\lambda(w)}|z| \quad \text{as well as} \quad \big(z\, (D^2F)^{-1}(w), z\big)\ge \frac{1}{\Lambda(w)}\, |z|^2.
	 \eeq
Therefore, thanks to \eqref{dgk}, \eqref{eigenv}, and \eqref{quc} we have
	\beq
	\label{dgk1}
	 |Dg_ k(z)|= \frac{|z \, (D^2F)^{-1}(w)|}{\big(z\, (D^2F)^{-1}(w), z\big)} \le \frac{|z| \Lambda(w)}{|z|^2 \lambda(w)}  \le \frac{H}{|z|} \qquad \text{on $\{G= k\}$}.
	 \eeq
Using \eqref{prog2} and \eqref{prog3} gives
\[
	|Dg_ k(z)|= \l|Dg_{ k}\l(\frac{z}{g_{ k}(z)}\r)\r|.
	\]
Therefore we can use \eqref{dgk1} with the vector $\displaystyle \frac{z}{g_k(z)}$, thus obtaining
	\[
	|Dg_ k(z)| \le H\, \frac{g_ k(z)}{|z|}\le H\, \sup_{\S^{N-1}} g_ k\qquad \forall \, z\ne 0.
	\]
This proves \eqref{Dgkappa}.
	To verify \eqref{supinfg}, notice that under the stated assumptions on $F$ the map $t\mapsto G(t\, z)$ is strictly increasing on $]0, +\infty[$, hence  $\{g_k=1\}=\{G=k\}$.
	  By the $1$-homogeneity of $g_k$, it  then holds
	 \[
	\sup_{\S^{N-1}}g_ k=\sup_{\{G= k\}}|z| \quad \text{as well as} \quad \inf_{\S^{N-1}}g_ k=\inf_{\{G= k\}}|z|.
	 \]
Since $\{G\le  k\}$ is star-shaped with respect to the origin, we can choose  $\bar x$ and $\bar y$ in $\{G=  k\}$ such that
	\[
	|\bar x|=\sup_{\{G=  k\}}|z| \quad \text{as well as} \quad |\bar y|=\inf_{\{G=k\}}|z|.
	\]
	By  \eqref{quasisymmetry} applied to $DF$ and \eqref{hypF} we have
	\beq
	\label{gas2}
\frac{\sup_{\S^{N-1}}g_ k}{\inf_{\S^{N-1}}g_ k}=\frac{|\bar x|}{|\bar y|}=\frac{|DF(DF^{-1}(\bar x))|}{|DF(DF^{-1}(\bar y))|}\le C\, \eta_H\left(\frac{|DF^{-1}(\bar x)|}{|DF^{-1}(\bar y)|}\right).
\eeq
By the definition of $G$, it holds $DF^{-1}(\bar x), \,  DF^{-1}(\bar y) \in \{F=  k\}$, hence \eqref{estF} gives
\[
1= \frac{F(DF^{-1}(\bar y))}{F(DF^{-1}(\bar x))}\le C\, \frac{|DF^{-1}(\bar y)|}{|DF^{-1}(\bar x)|}\, \eta_H\left(\frac{|DF^{-1}(\bar y)|}{|DF^{-1}(\bar x)|}\right)=C\, \eta_{1+H, 1+1/H}\left(\frac{|DF^{-1}(\bar y)|}{|DF^{-1}(\bar x)|}\right)
\]
so that
\[
\frac{|DF^{-1}(\bar x)|}{|DF^{-1}(\bar y)|}\le \frac{1}{\eta_{1+H, 1+1/H}^{-1}(1/C)}.
\]
Inserting this estimate into \eqref{gas2} gives  us the conclusion.
%\[
%\eta^{-1}\left(|DF^{-1}(\bar x)|\right)\, |DF^{-1}(\bar x)| \le C\, k, \qquad k\le C\, \eta\left(|DF^{-1}(\bar y)|\right) \, |DF^{-1}(\bar y)|
%\]
%so that
%
%
%\eqref{roundsections} (SOSTITUIRE e USARE \eqref{twosidedF}) yields
%	\[
%	|DF^{-1}(\bar x)|\le \rho( k) \quad \text{as well as} \quad |DF^{-1}(\bar y)|\ge \tau\, \rho( k),
%	\]
%and in particular 
%	\beq
%	\label{ineq-quot}
%	\frac{|DF^{-1}(\bar x)|}{|DF^{-1}(\bar y)|}\le\frac{1}{\tau}.
%	\eeq
%Finally, applying \eqref{quasisymmetry} to $z=DF^{-1}(\bar x)$, $w=DF^{-1}(\bar y)$, and $z_0=0$, using \eqref{ineq-quot}, Lemma \ref{eta-lemma}, and the monotonicity of $\eta$ gives
%	\[
%	\frac{|\bar x|}{|\bar y|}\le \eta(C/\tau),
%	\]
%proving the claim that in particular reads as
%	\[
%	\xi(k) \le \eta(C/\tau).
%	\]

Finally, to prove the last assertion, let us fix $k>h$, choose $\tilde{x}$ such that
\[
G(\tilde x)=k  \quad \text{as well as} \quad g_h(\tilde x)=\inf_{G\ge k} g_h, 
\]
and set
\[
\tilde{y}=\frac{\tilde x}{g_h(\tilde x)}.
\]
Then in holds $g_h(\tilde y)=1$, so that $\tilde y\in \{G=h\}$. Set furthermore
\beq
\label{hk0}
x:=DF^{-1}(\tilde x)\in \{F= k\} \quad \text{as well as} \quad y:=DF^{-1}(\tilde y)\in \{F=h\}.
	\eeq
%We first show that
%	\beq
%	\label{claim1}
%	\inf _{\{G\ge  k\}}(g_h-1) \ge C \l(\frac{h}{ k}\r)^{\frac{H^3}{H+1}} \eta^{-1} \l(\frac{ k}{h}-1\r),
%	\eeq
%for some $C>0$. 
By the definition of $\tilde x$ and $\tilde y$ and the $1$-homogeneity of $g_h$ it follows that
	\beq
	\label{hk1.1}
	\inf _{\{G\ge  k\}}(g_h-1)= \frac{g_h(\tilde x)-g_h(\tilde y)}{g_h(\tilde y)} {=\frac{|\tilde x-\tilde y|}{|\tilde y|}.}
	\eeq
Using the distortion estimate \eqref{quasisymmetry} to the map $DF^{-1}$ gives
	\[
	\frac{|x-y|}{|y|}\le C \, \eta_H\left(\frac{|\tilde x-\tilde y|}{|\tilde y|}\right)
	\]
 so that, by \eqref{dis-eta} for the function $\eta_H^{-1}$ we obtain
	\[
	\frac{|\tilde x- \tilde y|}{|\tilde y|} \ge \eta_H^{-1} \l(\frac{1}{C} \frac{|x-y|}{|y|}\r) \ge \frac{1}{C}\,  \eta_H^{-1} \l(\frac{|x-y|}{|y|}\r), 
	\]
which inserted into \eqref{hk1.1} reads as
\beq
	\label{step1}
	\inf_{\{G \ge  k\}} (g_h-1) \ge \frac{1}{C}\,  \eta^{-1}_H \l(\frac{|x-y|}{|y|}\r).
	\eeq
By \eqref{hk0} we have
	\beq
	\label{step1.1}
	\begin{split}
	 k-h&=F(x)-F(y)\le \int_0^1 |DF(y+t\, (x-y))|\, |x-y|\, dt\\
	 &\le |x-y|\, \left[\int_0^1 |DF(y+t\, (x-y))-DF(y)|\, dt+|DF(y)|\right].
	 \end{split}
	\eeq
To estimate the rightmost integral in the above inequality we use \eqref{quasisymmetry} for the map $DF$  (recalling that $DF(0)=0$) and \eqref{basiceta} to have
	\[
	\begin{split}
	\int_0^1 |DF(y+t\, (x-y))-DF(y)|\, dt&\le C\, |DF(y)|\, \int_0^1\eta_H\left(t\, \frac{|x-y|}{|y|}\right)\, dx\\
	&\le C\, |DF(y)|\, \eta_H\left(\frac{|x-y|}{|y|}\right)\, \int_0^1\eta_H(t)\, dt\\
	&\le \frac{CH}{H+1}\, |DF(y)|\, \eta_H\left(\frac{|x-y|}{|y|}\right).
	\end{split}
	\]
	Therefore \eqref{step1.1} simplifies to
	\[
	k-h\le C\, |DF(y)|\, |x-y|\, \left[1+\eta_H\left(\frac{|x-y|}{|y|}\right)\right].
	\]
	On the other hand,  \eqref{deltamon2} implies
	\[
	h=F(y) \ge    |y|\, |DF(y)|/C.
	\]
Gathering the previous two inequalities gives
	\beq
	\label{step1.12}
	\frac{ k-h}{h}\le C\,  \frac{|x-y|}{|y|}\,\left[1+\eta_H\left(\frac{|x-y|}{|y|}\right)\right].  
	\eeq
Since 
\[
 t\,  (1+\eta_H(t))\le 2\, \eta_{1, H+1}(t),
 \]
we can rewrite \eqref{step1.12} as
\[
\frac{|x-y|}{|y|}\ge \eta_{1, H+1}^{-1}\l(\frac{k-h}{C\, h}\r)
\]
and recalling \eqref{step1} gives
\[
\inf_{\{G \ge  k\}} (g_h-1) \ge \frac{1}{C}\,  \eta^{-1}_H \l(\eta_{1, H+1}^{-1}\l(\frac{k-h}{ h}\r)\r)=\frac{1}{C} \, \eta^{-1}_{1/H, H(H+1)}\l(\frac{k-h}{ h}\r),
\]
where we used \eqref{dis-eta}, \eqref{compo} to clean up the estimate.
Using \eqref{eta-inv} gives \eqref{claim1} and then the conclusion. 
%	\[
%	\frac{|y|}{|x|} \ge \psi\left(\frac{h}{C  k}\right),
%	\]
%which  gives
%	\[
%	\frac{|x|}{|y|} \le \frac{1}{\psi\left(\frac{h}{C  k}\right)} \le \beta\l(C \frac{ k}{h}\r),
%	\]
%where
%	\[
%	\beta(t)= \max\l\{t^{\frac{1}{H+1}}, t^{\frac{H}{H+1}}\r\}.
%	\]
%From \eqref{hk2} and using the monotonicity of $\eta$ we have
%	\beq
%	\label{dis-beta}
%	\frac{ k- h}{h} \le C \eta \l(\beta\l(C \frac{ k}{h}\r)\r) \frac{|x-y|}{|y|}.
%	\eeq
%We now distinguish between two cases. If $C \frac{ k}{h}< 1$, thanks to Lemma \ref{eta-lemma} and the fact that $ k>h$ we have
%	\[
%	C \eta \l(\beta\l(C \frac{ k}{h}\r)\r)= C \eta \l(\l(C \frac{ k}{h}\r)^{\frac{1}{H+1}}\r) \le C \eta\l(\l(\frac{ k}{h}\r)^{\frac{1}{H+1}}\r)= C \l(\frac{ k}{h}\r)^{\frac{H}{H+1}}.
%	\]
%On the other hand, if $\displaystyle C\frac{ k}{h}\ge 1$, a similar argument leads us to
%	\[
%	C \eta \l(\beta\l(C \frac{ k}{h}\r)\r) \frac{|x-y|}{|y|} \le C \l(\frac{ k}{h}\r)^{\frac{H^2}{H+1}},
%	\]
%therefore in any case from \eqref{dis-beta} we can conclude that
%	\[
%	{\frac{ k-h}{h} \le C\, \left(\frac{ k}{h}\right)^{\frac{H^2}{H+1}} \, \frac{|x-y|}{|y|}.}
%	\]
%This implies that
%	\[
%	\frac{|x-y|}{|y|} \ge C \l(\frac{h}{ k}\r)^{\frac{H^2}{H+1}} \frac{ k-h}{h},
%	\]
%
%Using the above inequality into  \eqref{step1} and exploiting Remark \ref{eta-rem} and inequality \eqref{eta-inv} give
%	\[
%	\inf _{\{G\ge  k\}}(g_h-1) \ge C \l(\frac{h}{ k}\r)^{\frac{H^3}{H+1}} \eta^{-1} \l(\frac{ k}{h}-1\r),
%	\]
%for some $C>0$, and therefore \eqref{claim1} follows. 
%
 \end{proof}
% 
% \begin{ex}
% Let 
% \[
% F(z)=\max\l\{\frac{1}{1+H}\, |z|^{1+H}, \frac{H}{1+H}\, |z|^{1+\frac{1}{H}}\r\}.
% \]
% Check that $F$ is $H$-q.\,u.\,c.\,and check optimality of \eqref{claim1}.
% \end{ex}
%

\section{Proof of the main result}
\label{7}
\begin{theorem}
Let $F$ be a $H$-q.\,u.\,c.\,function obeying \eqref{hypF} and let $u$ be a local  minimiser for $J$ in $\Omega$ in the sense of Definition \ref{def:local-min}. Then, there exists a constant $C=C(H, N)>0$ such that,
 for any $B_R$ with $B_{2R}\subseteq \Omega$, $u$ is locally Lipschitz and satisfies the estimate
	\beq
	\label{stimafinale}
	\sup_{B_{R/2}}|F(Du)|\le C\,  \intmed_{B_{2R}}F(Du)\, dx.
	\eeq
	\end{theorem}

\begin{proof}
Let us suppose for the moment that {both $F$ and $u$ are smooth in $B_{2R}$.}
Moreover, by considering $F/i_F$ (see \eqref{defif}), we can also suppose  that
\beq
\label{iffin}
\inf_{|z|=1} |DF(z)|=1.
\eeq

Fix $R>r>0$ and set $\bar R=\dfrac{R+r}{2}$. Choose $\varphi\in C^\infty_c(B_{\bar R}, [0,1])$  such that 
\beq
\label{prop-varphi}
{\rm supp}\, (\varphi)\subseteq B_{\bar R}, \qquad \varphi\lfloor_{B_r}\equiv 1, \qquad  |D\varphi|\le \frac{C}{R-r}.
\eeq
%Let
%\[
%G(z)=F(DF^{-1}(z))
%\]
%which is well defined and smooth and 
For any $k\ge 0$ we consider the Minkowski functional of $\{G \le k\}$ defined as in \eqref{minkowski}, with $G$ given by \eqref{G}, which satisfies by construction 
\beq
\label{AG2}
g_k(z)\le 1\quad \Longleftrightarrow \quad \text{ $z\in \{G\le k\}$}\quad \Longleftrightarrow \quad F(DF^{-1}(z))\le k.
\eeq
Let $V$ be given by \eqref{def-V}. It holds  $\varphi \, (g_ k(V)-1)_+\in W^{1,2}_0(B_{\bar R})$ and thanks to \eqref{AG2} 
\[
(g_k(V)-1)_+\equiv 0 \qquad \text{ on }   B_{\bar R}\setminus A\left( k, \bar R\right),
\]
 where we set
\[
A(k, \rho)=\{x\in B_\rho: F(Du(x))\ge k\}=\{x\in B_\rho:G(V(x))\ge k\},
\]
i.\,e.\,equation \eqref{A-kr} for the affine function $\ell(z)\equiv k$.  Then the Sobolev embedding and the chain rule give
	\begin{equation}
	\label{ineq}
	\begin{split}
	 \left(\int_{A( k, r)}|g_ k(V)-1|^{2^*}\, dx\right)^{\frac{2}{2^*}}&  \le  \left(\int_{B_{\bar R}}|\varphi\, (g_ k(V)-1)_+|^{2^*}\, dx\right)^{\frac{2}{2^*}} \\
	&\le C \int_{B_{\bar R}}|D(\varphi\, (g_ k(V)-1)_+)|^{2}\, dx\\
	&\le  C\int_{B_{\bar R}}|(g_ k(V)-1)_+|^{2}|D\varphi|^2\, dx\\
	&\quad +C \int_{B_{\bar R}}|D(g_ k(V)-1)_+)|^{2}\varphi^2\, dx.
	\end{split}
	\end{equation}
In order to estimate the above right-hand side,  for the first term we use \eqref{AG2} and \eqref{prop-varphi} to achieve
	\beq
	\label{temp3.0}
	\int_{B_{\bar R}} |(g_{ k}(V)-1)_+|^2 |D\varphi|^2 \, dx   \le \frac{C}{(R-r)^2} \int_{A( k, R)} (g_{ k}(V)-1)^2 \; dx,
	\eeq
while for the second term we apply the chain rule and inequalities \eqref{AG2} and \eqref{Dgkappa}  to have
\beq
\label{temp3.00}
\begin{split}
	\int_{B_{\bar R}}|D(g_ k(V)-1)_+)|^{2}\varphi^2\, dx \, & { \le \int_{A( k, \bar R)} |D g_ k(V)|^2 |DV|_2^{2}\, dx} \\
	&\le H^2\, \sup_{\S^{N-1}} g_ k^2\int_{A( k, \bar R)}|DV|_2^2\, dx.
	\end{split}
\eeq
To further estimate  the last integral, we use the Caccioppoli inequality \eqref{cac} with $\rho=\bar R< R$, which  results in
	 \beq
	 \label{temp3.1}
%	\int_{B_{\bar R}}|D(g_ k(V)-1)_+)|^{2}\varphi^2\, dx
	H^2\, \sup_{\S^{N-1}} g_ k^2\int_{A( k, \bar R)}|DV|_2^2\, dx \le \frac{C}{(R-r)^2}\,  \sup_{\S^{N-1}} g_k^2\, \int_{A( k, R)} |V|^2\, dx.
	 \eeq
Moreover, we observe that $g_k(V)\ge 1$ on $A( k, R)$, hence the $1$-homogeneity of $g_k$ implies that
\beq
	\label{temp3.11}
	 |V|=\frac{|V|}{g_k(V)}\, g_k(V)  \le    \sup_{z\ne 0} \frac{|z|}{g_{ k}(z)} \, g_{ k}(V) \le \frac{g_ k (V)}{\inf_{\S^{N-1}} g_ k}.
\eeq
From \eqref{temp3.1}, \eqref{temp3.11} and  \eqref{supinfg} we have that  \eqref{temp3.00}   simplifies to 
	 \[
	 \int_{B_{\bar R}}|D(g_ k(V)-1)_+)|^{2}\varphi^2\, dx\le  \frac{C}{(R-r)^2}\int_{A( k, R)} g^2_k(V)\, dx,
	  \]
	  which, inserted into \eqref{ineq} together with \eqref{temp3.0}, gives
	 \beq
	 \label{dg4}
	 \begin{split}
	 \left(\int_{A( k, r)}|g_ k(V)-1|^{2^*}\, dx\right)^{\frac{2}{2^*}}&\le \frac{C}{(R-r)^2} \int_{A( k,  R)}  g_ k^2(V)+(g_ k(V)-1)^2\, dx\\
	 &\le \frac{C}{(R-r)^2}\left(\int_{A( k,  R)} (g_ k(V)-1)^2\, dx + {|A( k, R)|} \right),
	 \end{split}
	 \eeq
 with a constant $C=C(H)$.
 
	 Let now $0<h<k$ and observe that
	\[
g_h(V(x))-1\ge \inf _{\{G\ge  k\}}(g_h-1) \qquad \text{for $x\in A(k, R)$},
	\]
	therefore \eqref{claim1} gives
	\[
	\frac{1}{C}\, \min\left\{\left(\frac{k}{h}-1\right)^H, \left(\frac{k}{h}-1\right)^{\frac{1}{H(H+1)}}\right\}^2 \, |A(k, R)|\le \int_{A(k, R)} (g_h(V)-1)^2\, dx,
	\]
	so that
	 \beq
	 \label{claim-A}
	|A( k, R)|\le \frac{C}{ \left(\frac{k}{h}-1\right)^{2H}}\, \int_{A( k, R)}(g_h(V)-1)^2\, dx\qquad \text{if}\quad  h<k<2\,h.
	\eeq
Inserting this estimate into \eqref{dg4} gives for $0<h<  k<2\, h$
	\beq
	\label{dg6}
	\begin{split}
	& \left(\int_{A( k, r)}|g_ k(V)-1|^{2^*}\, dx\right)^{\frac{2}{2^*}} \\
	&\le \frac{C}{(R-r)^2} \int_{A( k,  R)} (g_ k(V)-1)^2 +\frac{C}{\left(\frac{k}{h}-1\right)^{2H}} \, (g_h(V)-1)^2\, dx\\
	&\le  \frac{C}{\left(\frac{k}{h}-1\right)^{2H}(R-r)^2} \int_{A(h,  R)} (g_h(V)-1)^2 \, dx,
	\end{split}
	\eeq
where in the second inequality we used the fact that
	\beq
	\label{kha}
	A( k, R)\subseteq A(h, R) \quad \text{as well as} \quad g_h\ge g_ k\quad \text{for all $h\le  k$}
	\eeq
 and the bound
 \[
 \left(\frac{k}{h}-1\right)^{-2H}\ge 1 \qquad \text{for $0<h<k<2h$}.
 \]
Estimate \eqref{dg6} in conjunction with \eqref{claim-A} allows for a  De Giorgi-type iteration, whose proof we include here for completeness. For $ 0<h<k<2h$, we apply H\"older's inequality, \eqref{dg6}, \eqref{claim-A} and  \eqref{kha} to have
	\beq
	\label{dgf}
	\begin{split}
	 \int_{A( k,  r)} (g_ k(V)-1)^2 \, dx & \le |A( k, r)|^{\frac{2}{N}} \left(\int_{A( k, r)}|g_ k(V)-1|^{2^*}\, dx\right)^{\frac{2}{2^*}}   \\
	 &\le  \frac{C\, |A(k, r)|^{\frac{2}{N}}}{\left(\frac{k}{h}-1\right)^{2H}(R-r)^2} \int_{A(h,  R)} (g_h(V)-1)^2 \, dx\\
	 &\le \frac{C }{{\left(\frac{k}{h}-1\right)^{2H(1+2/N)}}(R-r)^2} \left(\int_{A(h,  R)} (g_h(V)-1)^2 \, dx\right)^{1+\frac{2}{N}}.
	 \end{split}
	 \eeq
Let $M>0$ to be chosen and set for $n\ge 0$
	\[
	 k_n= M\, (1-2^{-(n+1)}),  \qquad r_n=\frac{R}{2}+\frac{R}{2^{n+1}}, \qquad X_n=\int_{A( k_n,  r_n)} (g_{ k_n}(V)-1)^2 \, dx,
	\]
so that $\lim_{n \to \infty} k_n= M$ and $\lim_{n\to \infty} r_n= R/2$. An elementary computation shows that
\[
k_n\le k_{n+1}\le 2\,  k_n, \qquad \frac{k_{n+1}}{k_n}-1\ge \frac{1}{2^{n+2}}, \qquad (r_{n}-r_{n+1})^2=\frac{R^2}{4^{n+2}},
\]
therefore \eqref{dgf} with the viable choice $k=k_{n+1}$ and $h=k_n$ reads as
\[
X_{n+1}\le \frac{C}{R^2} \, b^n\, X_{n}^{1+2/N},
\]
for some positive  $b= b(H, N)$ and $C=C(H, N)$.

Thanks to \cite[Lemma 7.1]{Giusti},  it holds
	\beq
	\label{dg9}
 	X_0\le \frac{R^N}{C^{N/2}\, b^{N^2/4}} \quad \Longrightarrow\quad   \lim_{n \to \infty} X_n=0.
	\eeq
Moreover, from $X_n\to 0$ it follows that
	\[
	\int_{A(M, R/2)}(g_M(V)-1)^2\, dx=0
	\]
which in turn by \eqref{AG2}   forces $F(Du)\le M$ a.\,e.\,in $B_{R/2}$.  
Let us use this information to write explicitly the implication in \eqref{dg9}. Since  $r_0=R$ and $ k_0=M/2$, it reads as
	\beq
	\label{fin}
 	\int_{A(M/2, R)} (g_{M/2}(V)-1)^2\, dx\le c\, R^N \quad \Longrightarrow\quad  \sup_{B_{R/2}} F(Du)\le M,
	\eeq
for some $c=c(H, N)>0$.
It remains to choose $M>0$ in such  a way that the first inequality in \eqref{fin} holds true.
We aim to obtain an upper bound for $g_{M/2}$ in terms of an upper bound for $G$. To get the latter,   we recall that by \eqref{hypF} the following convexity inequality
	\[
	F(z)\le\left( DF(z), z\right) \qquad \forall \, z\in \R^N
	\]
	holds true.
We choose  $z=DF^{-1}(y)$ and then use the quasisymmetry \eqref{quasisymmetry} of $DF^{-1}$ to obtain
	\[
	G(y)\le \left(y,  DF^{-1}(y)\right)\le |DF^{-1}(y)|\, |y| \le   C  \, |DF^{-1}(w)|\, |y|\, \eta_H\left(\frac{|y|}{|w|}\right),
	\]
 for some $C=C(H, N)>0$ and for all $y\in \R^N$,  $w\in \R^N\setminus\{0\}$.
Choose $w=DF(z)$ for arbitrary $z$ such that  $|z|=1$  to get
 \[
 G(y)\le C\, |y|\, \inf_{|z|=1}\eta_H\left(\frac{|y|}{|DF(z)|}\right).
\]
Moreover,  the normalisation \eqref{iffin} implies that $1/|DF(z)| \le 1$ for all $z$ such that $|z|= 1$ which, together with the monotonicity of $\eta_H$, implies that the previous inequality reads as
\[
G(y)\le C\,|y|\,  \eta_H(|y|)=C\, \eta_{H+1, 1+\frac{1}{H}}(|y|)\qquad \forall \, y\in \R^N.
\]
Choosing  $\displaystyle y= \frac{z}{g_{M/2}(z)}$  yields, by the very definition of $g_{M/2}$, 
	\[
	\frac{M}{2}=G\left(\frac{z}{g_{M/2}(z)}\right)\le C \,  \eta_{H+1, 1+\frac{1}{H}}\left(\frac{|z|}{g_{M/2}(z)}\right),
	\]
which, solved with respect to $g_{M/2}(z)$, gives through   \eqref{eta01} and \eqref{dis-eta}
	\[
	g_{M /2}(z)\le \frac{C \, |z|}{\eta^{-1}_{H+1, 1+\frac{1}{H}}(M)}\le C\, |z|\,\eta_{\frac{1}{H+1}, \frac{H}{H+1}}(1/M).
	\]
	Therefore
	\[
	\int_{A(M/2, R)} (g_{M/2}(V)-1)^2\, dx\le \int_{B_R} g_{M/2}^2(V)\, dx\le  C\, \eta_{\frac{1}{H+1}, \frac{H}{H+1}}^2(1/M)\int_{B_R}|V|^2\, dx.
	\]
	In order to satisfy the assumption of \eqref{fin}, we can then choose $M$ so that 
	\[
	C\, \eta_{\frac{1}{H+1}, \frac{H}{H+1}}^2(1/M)\int_{B_R}|V|^2\, dx\le c\, R^N, 
	\]
	i.\,e., rearranging and using \eqref{eta01} again,
	\[
	\left(\intmed_{B_R}|V|^2\, dx\right)^{\frac{1}{2}}\le \frac{c}{\eta_{\frac{1}{H+1}, \frac{H}{H+1}}(1/M)}=c\, \eta_{H+1, 1+\frac{1}{H}}^{-1}(M)
	\]
	for some $c=c(H, N)>0$, which is readily solved through \eqref{dis-eta} as
	\[
	M \ge  C\, \eta_{H+1, 1+\frac{1}{H}}\left(\left(\intmed_{B_R}|V|^2\, dx\right)^{\frac{1}{2}}\right).
	\]
	Now recall that by \eqref{proso} it holds
	\[
	\left(\intmed_{B_R}|V|^2\, dx\right)^{\frac{1}{2}}\le C\, \eta_{\frac{H}{H+1}, \frac{1}{H+1}}\left(\intmed_{B_{2R}} F(Du)\, dx\right).
	\]
Then, thanks to \eqref{compo} and again \eqref{dis-eta} it is sufficient to require
	\beq
	\label{dis-M}
	M\ge C\, \eta_{H+1, 1+\frac{1}{H}}\circ\eta_{\frac{H}{H+1}, \frac{1}{H+1}}\left(\intmed_{B_{2R}} F(Du)\, dx\right)=C\, \eta_H\left(\intmed_{B_{2R}} F(Du)\, dx\right).
	\eeq
	Choosing $M$ such that the equality in \eqref{dis-M} holds true gives, thanks to \eqref{fin}, that $F(Du)\le M$ in $B_{R/2}$, with the  estimate
	\[
	\sup_{B_{R/2}}|F(Du)|\le C\, \eta_H \left(  \intmed_{B_{2R}}F(Du)\, dx\right).
	\]
	We can then remove assumption \eqref{iffin}  to get, for a general $i_F>0$,
	\beq
	\label{fin00}
	\sup_{B_{R/2}}|F(Du)|\le C\, i_F\, \eta_H \left(\frac{1}{i_F}\, \intmed_{B_{2R}}F(Du)\, dx\right).
	\eeq
	We next take advantage of the invariances of this estimate. For any $\lambda>0$, the function $u_\lambda(x)=\lambda\, u(x)$ is a local minimiser in $\Omega$ for the functional $J_\lambda$ having integrand 
	\[
	F_\lambda(z)=F(z/\lambda).
	\]
In particular, \eqref{fin00} holds true also for such $F_{\lambda}, u_{\lambda}$, that is,
	\beq
	\label{fin001}
	\sup_{B_{R/2}}|F_{\lambda}(Du_{\lambda})|\le C\, i_{F_{\lambda}}\, \eta_H \left(\frac{1}{i_{F_{\lambda}}}\, \intmed_{B_{2R}}F_{\lambda}(Du_{\lambda})\, dx\right).
	\eeq
	Notice that $F_\lambda$ is still $H$-q.\,u.\,c.\,and  that
\beq
\label{iflambda}
i_{F_{\lambda}}=\frac{1}{ \lambda}\, \inf_{|z|=1}\left|DF\Big(\frac{z}{\lambda}\Big)\right|.
\eeq
For any $z, w$ of unit norm, the distortion estimate \eqref{quasisymmetry} applied first to the vectors $z/\lambda, w$ and then to $w, z/\lambda$ gives
\[
\left|DF\Big(\frac{z}{\lambda}\Big)\right|\le C\, |DF(w)|\, \eta_H(1/\lambda) \quad \text{as well as} \quad |DF(w)|\le C\, \left|DF\Big(\frac{z}{\lambda}\Big)\right|\, \eta_H(\lambda),
\]
respectively, so that 
\[
\frac{i_F}{C\, \eta_H(\lambda)}\le \left|DF\Big(\frac{z}{\lambda}\Big)\right|\le C\, i_F \, \eta_H(1/\lambda)
\]
for all $|z|=1$.  Taking the infimum for such $z$'s and using \eqref{iflambda} give
\[
\frac{i_F}{C\, \eta_H(\lambda)\, \lambda}\le i_{F_\lambda}\le C\, i_F \, \frac{\eta_H(1/\lambda)}{\lambda}.
\]
Since $DF$ is $H^{N-1}$-quasiconformal, it is also locally $1/H$-H\"older continuous and so is the map $t\mapsto |DF(t\, z)|$. It follows that  $\lambda\mapsto i_{F_\lambda}$ is continuous on $]0, +\infty[$, being an infimum of equicontinuous functions.  Moreover, it fulfils
\[
\begin{split}
\lim_{\lambda\downarrow 0} i_{F_\lambda}&\ge  \frac{ i_F}{C}\, \lim_{\lambda\downarrow 0} \frac{1}{\eta_H(\lambda)\, \lambda}=\infty \quad \text{as well as} \quad
\lim_{\lambda\uparrow \infty} i_{F_\lambda}\le  C\, i_F\, \lim_{\lambda\uparrow \infty}  \frac{\eta_H(1/\lambda)}{\lambda}=0.
\end{split}
\]
Thanks to the intermediate value theorem we can then choose $\lambda$ so that 
\[
\frac{1}{i_{F_{\lambda}}} \, \intmed_{B_{2R}} F_{\lambda}(Du_{\lambda})\, dx=\frac{1}{i_{F_{\lambda}}} \, \intmed_{B_{2R}} F(Du)\, dx=1.
\]
For such  $\lambda$ and corresponding $u_\lambda$, $F_\lambda$, \eqref{fin001} becomes \eqref{stimafinale}, since
\[
 \sup_{B_{R/2}} F_\lambda(Du_\lambda)=\sup_{B_{R/2}} F(Du) \quad \text{as well as} \quad 
 i_{F_\lambda}=\intmed_{B_{2R}} F(Du)\, dx,
 \]
 and  
 \[
 \eta_H\left(\frac{1}{i_{F_{\lambda}}} \, \intmed_{B_{2R}} F_{\lambda}(Du_{\lambda})\, dx\right)=\eta_H(1)=1.
 \]
This gives us the conclusion in the smooth case.

To remove the smoothness assumption on $F$ and $u$, let $F_n$ and $u_n$ be the regularised approximations obtained thanks to Lemma \ref{regularized}. Then we can apply the same argument as before to $F_n, u_n$, thus obtaining
	\beq
	\label{stimafinale1}
	\sup_{B_{R/2}} |F_n(Du_n)| \le C\,  \intmed_{B_{2R}} F_n(Du_n) \; dx.
	\eeq
Now observe that \eqref{prosobe} allows to assume   (up to subsequences) that $DF_n(Du_n)\to DF(Du)$ pointwise a.\,e. in $B_{2R}$, while Proposition \ref{proqu}-{\em (4)} ensures that $DF_n^{-1}\to DF^{-1}$ locally uniformly, so that both $Du_n\to Du$ and $F(Du_n)\to F(Du)$, pointwise a.\,e. in $B_{R}$. This observation and the first estimate in \eqref{prosobe} allow to pass to the limit in \eqref{stimafinale1}, thus obtaining again the conclusion.
 \end{proof}% 

\end{document}